\documentclass[11pt]{article}

\usepackage{bbm}
\usepackage{mathrsfs}
\usepackage{amsmath,amsfonts,amssymb,enumerate}
\usepackage{color}
\usepackage{lscape}
\usepackage{latexsym}
\usepackage{multirow}
\usepackage{rotating}
\usepackage{threeparttable}
\usepackage{graphicx}
\usepackage{algpseudocode,algorithm}
\usepackage[margin=1in]{geometry}
\usepackage{appendix}
\DeclareMathOperator{\dist}{dist}
\DeclareMathOperator{\kernel}{ker}
\DeclareMathOperator{\range}{Ran}
\def \supp {\,{\rm supp}\,}
\def \prox {\,{\rm prox}\,}
\def \argmin {\,{\rm argmin}\,}
\numberwithin{equation}{section}
\newtheorem{theorem}{\bf Theorem}[section]
\newtheorem{lemma}[theorem]{\bf Lemma}
\newtheorem{proposition}[theorem]{\bf Proposition}

\newtheorem{coro}[theorem]{\bf Corollary}

\newenvironment{proof}{\noindent{\em Proof:}}{\quad \hfill$\Box$\vspace{2ex}}

\renewcommand{\theequation}{\arabic{section}.\arabic{equation}}

\newcommand{\vertiii}[1]{{\left\vert\kern-0.25ex\left\vert\kern-0.25ex\left\vert #1 
    \right\vert\kern-0.25ex\right\vert\kern-0.25ex\right\vert}} 

\graphicspath{{figs/}}

\begin{document}

\title{\bf A Duality Approach to Regularized Learning Problems in Banach Spaces}

\author{Raymond Cheng\thanks{
Department of Mathematics and Statistics, Old Dominion University, Norfolk, VA 23529, USA. E-mail address: {\it rcheng@odu.edu}.}, \quad Rui Wang\thanks{School of Mathematics, Jilin University, Changchun 130012, P. R. China. E-mail address: {\it rwang11@jlu.edu.cn}.} \quad and \quad Yuesheng Xu\thanks{Department of Mathematics and Statistics, Old Dominion University, Norfolk, VA 23529, USA. This author is also a Professor Emeritus of Mathematics, Syracuse University, Syracuse, NY 13244, USA. E-mail address: {\it y1xu@odu.edu.} All correspondence should be sent to this author.} }

\date{}

\maketitle{}

\begin{abstract}
Learning methods in Banach spaces are often formulated as regularization problems which minimize the sum of a data fidelity term in a Banach norm and a regularization term in another Banach norm. Due to the infinite dimensional nature of the space, solving such regularization problems is challenging. We construct a direct sum space based on the Banach spaces for the data fidelity term and the regularization term, and then recast the objective function as the norm of a suitable quotient space of the direct sum space. In this way, we express the original regularized problem as an unregularized problem on the direct sum space, which is in turn reformulated as a dual optimization problem in the dual space of the direct sum space. The dual problem is to find the maximum of a linear function on a convex polytope, which may be solved by linear programming. A solution of the original problem is then obtained by using related extremal properties of norming functionals from a solution of the dual problem. Numerical experiments are included to demonstrate that the proposed duality approach leads to an implementable numerical method for solving the regularization learning problems.
\end{abstract}



\section{Introduction}
Learning in Banach spaces has received considerable attention in the last two decades \cite{Ar, Az, Candes, Donoho, HLTY, PN, SKHK, Song-Zhang, Song-Zhang-Hickernell, Sri, Unser2, Unser1, WX2020, Xu2023, XuYe, Zhang-Xu-Zhang}. Learning methods are often formulated as regularization problems \cite{LP13, SKHK}. Such a problem seeks to minimize an objective function in the form of the sum of a data fidelity term in the norm of a Banach space, and a regularization term in the norm of another Banach space, which is often of infinite dimension. Due to the big data nature of recent real-world applications,  regularization in a Banach space with a sparsity promoting norm is widely used in various practical fields such as statistics \cite{tibshirani1996regression,tibshirani2011solution}, machine learning \cite{bi2003dimensionality,scholkopf2002learning}, signal processing \cite{chen2001atomic}, image processing and medical imaging \cite{MSX,rudin1992nonlinear}. In particular, special Banach spaces related to the $\ell_1$ space have proved useful in handwritten digit recognition \cite{lin2021multi, LWXY}.
However, effective solutions of these regularization problems are challenging. Especially, in the context of sparse learning \cite{Xu2023}, the norm of the Banach space for the regularization term is often chosen as a sparsity promoting norm, which usually is non-differentiable. Solving regularization problems having non-differentiable objective functions is even more troublesome, requiring great care. Moreover, it is demanding to develop efficient numerical solvers to learn a function in an infinite dimensional Banach space. 

The goal of this paper is to develop a duality approach for the solution of the class of regularization problems described above that arise in machine learning, attempting to provide a mathematical basis for further development of efficient numerical methods.  
Motivated by the duality approach \cite{CX1} for solving the minimum norm interpolation problem in $\ell_1(\mathbb{N})$, we develop a duality approach to solve problems of this type. Specifically, we construct a {\it direct sum space} based on the Banach spaces for the data fidelity term and for the regularization term, and identify the objective function of the regularization problem as the norm of a suitable quotient space of the direct sum space. This expresses the original regularization problem as a minimum norm problem on the quotient space. By considering the dual space of this resulting space, we further reformulate the original regularization problem as an equivalent problem in the dual space. By analyzing the geometry of the resulting space to determine the combination of parameters that will give rise to sparse solutions, we identify that the dual problem is to find the maximum of a linear function on a convex polytope. The resulting problem can be solved by linear programming. Once that problem is solved, we solve the original problem by using related extremal properties of norming functionals.

We proceed to describe the regularization problem to be considered in this paper. Let $\mathscr{X}$ be a real Banach space with the dual space $\mathscr{X}^*$. For $m\in\mathbb{N}$, we set $\mathbb{N}_m:=\{1,2,\ldots,m\}$. Suppose that $\mathbf{a}_k$, $k\in\mathbb{N}_m$, are linearly independent elements of $\mathscr{X}^*$, and $y_k$, $k\in\mathbb{N}_m$, are fixed real numbers, not all zero. Assume that there is a positive parameter $\rho$. Consider the minimization problem
\begin{equation}\label{genregprob0}
   \inf\Bigg\{ \sum_{k\in\mathbb{N}_m} \big| y_k - \langle \mathbf{x}, \mathbf{a}_k\rangle_{\mathscr{X}} \big|  + \rho\|\mathbf{x}\|_{\mathscr{X}}: \mathbf{x} \in \mathscr{X} \Bigg\}.
\end{equation}
The pairs $(\mathbf{a}_k, y_k)$, $k\in\mathbb{N}_m$, constitute a {\em training sample}.   The vector $\mathbf{x}$ is a {\em hypothesis}.   The quantity in \eqref{genregprob0} being minimized is known as a {\em loss functional} resulting from a hypothesis $\mathbf{x}$ and the training sample. 
This is a basic problem in machine learning. The vector $\mathbf{x}$ 
represents an estimate of the law underlying the training sample.  The second term is the ``regularizer,'' intended to prevent overfitting the data. The positive parameter $\rho$ controls the relative influence of this term. A related problem is to find $\mathbf{x} \in \mathscr{X}$  that minimizes
\begin{equation}\label{genintprob0}
    \inf \Big\{ \|\mathbf{x}\|_{\mathscr{X}}: \langle \mathbf{x}, \mathbf{a}_k\rangle_{\mathscr{X}}= y_k,\ k\in\mathbb{N}_m,\ \mathbf{x} \in \mathscr{X} \Big\}.
\end{equation}
This is to find the minimum norm $\mathbf{x}$ that  interpolates the data set. For this reason, \eqref{genintprob0} is called a minimum norm interpolation problem. Typical examples of the space $\mathscr{X}$ are the spaces $\ell_p(\mathbb{N})$, for $1\leq p\leq +\infty$.
Besides their theoretical importance, the spaces $\ell_p(\mathbb{N})$ have numerous applications, for instance, to signal processing and control theory \cite{ZSZ} and relaxation methods in linear programming \cite{VZ}.  When $1<p < 2$, the space $\ell_p(\mathbb{N})$ can be used to study the prediction of filtered symmetric-$\alpha$-stable processes; these are useful for modeling heavy-tailed phenomena \cite{Ch,CR,CR2,Kur}.  The case $p=1$ gives rise to a Banach algebra, as well as a sparsity-promoting geometry \cite{CX1}.  It has a unique importance in spectral factorization methods in signal processing \cite{BP}.

The regularization problem \eqref{genregprob0} may be viewed as a special case of a general problem. Suppose that $\mathscr{X}$, $\mathscr{Y}$ and $\mathscr{Z}$ are three real Banach spaces with dual spaces $\mathscr{X}^*,\mathscr{Y}^*$ and $\mathscr{Z}^*$, respectively, and there are constants $\rho>0$ and $1 \leq p < \infty$.  Let $\mathbf{y}_0 \in \mathscr{Y}$ be a fixed nonzero vector. Suppose that $A$ is a bounded linear operator from $\mathscr{X}$ to $\mathscr{Y}$ and $B$ is a bounded linear operator from $\mathscr{X}$ to $\mathscr{Z}$.
Consider the problem 

\begin{equation}\label{regprob}
   \inf\big\{ (\| {\mathbf{y}_0} - A{\mathbf{x}} \|^p_{\mathscr{Y}} + \rho\|B\mathbf{x}\|^p_{\mathscr{Z}})^{1/p} :\  \mathbf{x} \in \mathscr{X} \big\},
\end{equation}
and the corresponding general interpolation problem 
\begin{equation}\label{interpolprobalt}
   \inf\big\{ \|{\mathbf{x}}\|_{\mathscr{X}} :\  A\mathbf{x} = \mathbf{y}_0,\ \mathbf{x} \in \mathscr{X}  \big\}
\end{equation}
(inclusion of the power $p$ in the latter problem would have no meaningful effect).
Many data science problems fall into the setting of problem \eqref{regprob}. The well-known compressed sensing problem \cite{Candes, Donoho} may be reformulated in the form of \eqref{regprob}. Regularized learning  \cite{CS, EPP, MXY, Zhang-Xu-Zhang}, $l_1$-sparse regularization \cite{LSX, XuYe} and regularization models for inverse problems \cite{WX2019, WX2020, WX2021} are special examples of problem \eqref{regprob}.

In our approach we will introduce a normed space that is closely related to this loss functional. Specifically, the associated loss functional itself becomes a norm on a Banach space built from $\mathscr{Y}$ and $\mathscr{Z}$ and the ``training sample,'' as incorporated into $\mathbf{y}_0$, $A$ and $B$.  This reformulates the extremal problem \eqref{regprob} into an apparently simpler one, namely, to find a vector of minimum length in a convex subset of the constructed Banach space.
A  duality argument is used, transforming the original extremal problem into an equivalent one 
in the dual Banach space.  In practice, the dual extremal problem often involves only finitely many free parameters, and thus yields to numerical methods.  A solution to the dual problem is then used to identify solutions to the original problem \eqref{regprob}, by a correspondence between their norming functionals. The  extremal problems \eqref{regprob} and \eqref{interpolprobalt} respectively reduce to the problems \eqref{genregprob0} and \eqref{genintprob0} above by choosing $p=1$, $\mathscr{Y} = \mathbb{R}^m$ (with suitable norms attached) and $\mathscr{Z}=\mathscr{X}$. Moreover, $A$ is the operator from $\mathscr{X}$ to $\mathscr{Y}$ such that the $k$th entry of $A\mathbf{x}$ is $\langle \mathbf{x}, \mathbf{a}_k\rangle_{\mathscr{X}}$, $k\in\mathbb{N}_m$, and $B$ is the identity operator on $\mathscr{X}$. 
The existence and uniqueness of solutions will be addressed in due course.


Our approach makes use of numerous standard tools from classical functional analysis.  These results are  stated in the course of their usage; for reference, their proofs are gathered together in an appendix. Throughout this paper, $\mathscr{A}$, $\mathscr{B}$, $\mathscr{X}$, $\mathscr{Y}$ and $\mathscr{Z}$ will denote a separable Banach space over the real scalars.  All subspaces are understood to be closed in the norm topology. If $p$ is a parameter satisfying $1 \leq p \leq \infty$, then $p'$ will denote its H\"{o}lder conjugate, so that $1/p + 1/p' = 1$.

This paper is organized as follows.  
In the section to follow we review classical results of functional analysis connecting an extremal problem in a Banach space and its dual extremal problem. The notion of the norming functional is also covered. In Section \ref{dirsumsec}, we develop a duality approach for solving the regularized extremal problem  \eqref{regprob}. We begin reviewing the concept of a direct sum of two Banach spaces, and a  family of norms that can be placed on the direct sum.   A duality argument is then used to recast the original extremal problem \eqref{regprob} into an equivalent dual problem.  The solution to the original problem is then obtained via norming functionals. 
In Section 4, we consider regularization in a specific Banach space $\ell_1(\mathbb{N})$ of sequences defined on $\mathbb{N}$. We show that the dual problem is indeed a finite dimensional optimization problem, which can be solved by linear programming. The resulting dual solution leads us to reformulate the original infinite optimization problem as an equivalent finite dimensional one. The later can be solved by numerical methods such as the fixed point proximity algorithm (FPPA) developed in \cite{LSXZ, MSX}.
We present three numerical experiments in Section 5 to demonstrate the effectiveness of the proposed duality approach. There follows an appendix containing proofs of standard results from functional analysis.

\section{A Duality Approach to Extremal Problems}
To prepare for developing a duality approach for solving the regularized extremal problems, we consider in this section an abstract extremal problem which may be taken as a  best approximation problem in a Banach space. By using the characterizations of a general best approximation problem in functional analysis and convex analysis,  we transform the extremal problem into an equivalent dual
problem. The relation between the solutions of the original problem and its dual problem is characterized by the norming functionals. We also consider a special case that the Banach spaces involving in the extremal problem have pre-duals.  

We begin with describing the extremal problem under investigation. Suppose that $\mathscr{A}$ and $\mathscr{B}$ are two real Banach spaces and $H$ is a bounded linear operator from $\mathscr{A}$ to $\mathscr{B}$. Let $\mathbf{b}_0\in\mathscr{B}$ be a fixed vector. We consider the extremal problem 
\begin{equation}\label{best_app}
\inf\left\{\|\mathbf{b}_0- H\mathbf{a} \|_{\mathscr{B}}:\ \mathbf{a} \in\mathscr{A} \right\}.
\end{equation}
By introducing a closed subspace $\mathscr{M}$ of $\mathscr{B}$ as
\begin{equation}\label{subspace}
\mathscr{M}:=\overline{\{ H\mathbf{a}:\ \mathbf{a}\in \mathscr{A} \}},
\end{equation}
the extremal problem \eqref{best_app} seeks a best approximation to $\mathbf{b}_0$ from the closed subspace $\mathscr{M}$. We will apply the theory of the best approximation to problem \eqref{best_app}. Most proofs of the results in this section are standard. For convenience to readers, we include the proofs in this section and the appendix.

We next establish the duality problem of problem \eqref{best_app} via a functional analytic approach. To this end, we recall some
notions in Banach spaces. The dual space $\mathscr{A}^*$ of a Banach space $\mathscr{A}$ is the collection of all continuous linear functionals on $\mathscr{A}$. 
For $\mathbf{a} \in \mathscr{A}$ and $\lambda\in\mathscr{A}^*$, we may write $\langle \mathbf{a}, \lambda \rangle_{\mathscr{A}}$ in place of $\lambda(\mathbf{a})$, to emphasize that the functional is operating on a vector in $\mathscr{A}$. We also recall that $\mathscr{A}^*$ is itself a Banach space under the norm
\begin{equation*}
\|\lambda\|_{\mathscr{A}^*}  
:=\sup \Bigg\{ \frac{|\langle\mathbf{a}, \lambda \rangle_{\mathscr{A}}| }{\|\mathbf{a}\|_{\mathscr{A}}} :\ \mathbf{a} \in\mathscr{A}\setminus\{ \mathbf{0}\}\Bigg\},\ \mbox{for all}\ \lambda\in\mathscr{A}^*.
\end{equation*}
We use the notation $\mathscr{N}^{\perp}$ to mean the annihilator of a subset $\mathscr{N}$ in $\mathscr{A}$, that is, the subspace of $\mathscr{A}^*$ given by
$\mathscr{N}^{\perp}:=\big\{\lambda\in \mathscr{A}^*:\ \langle \mathbf{a}, \lambda \rangle_{\mathscr{A}} = 0,\ \ \mbox{for all}\ \ \mathbf{a}\in\mathscr{N}\big\}.$
If $\mathscr{N}$ is a closed subspace of $\mathscr{A}$, then the quotient space $\mathscr{A}/\mathscr{N}$ is defined to be the collection of cosets $\mathbf{a}+\mathscr{N}$, for all $\mathbf{a} \in \mathscr{A}$. The quotient space is also a Banach space when endowed with the norm
$\|\mathbf{a}+\mathscr{N}\|_{\mathscr{A}/\mathscr{N}}:= \inf\left\{\|\mathbf{a}-\mathbf{c}\|_{\mathscr{A}}:\ \mathbf{c} \in \mathscr{N}\right\}
   =\dist_{\mathscr{A}}(\mathbf{a}, \mathscr{N}).$
We denote by $\kernel T$ the kernel 
$\kernel T:=\{\mathbf{a}\in\mathscr{A}: T\mathbf{a}=0\}$
of the linear operator $T:\mathscr{A}\rightarrow\mathscr{B}$ and by $\range T$ the range $\range T:=\{\mathbf{b}\in\mathscr{B}: \mathbf{b}=T\mathbf{a},\ \mathbf{a}\in\mathscr{A}\}$
of the operator $T$.

The following lemma identifies the dual of a closed subspace of a Banach space, and the dual of a quotient space. For its proof we refer to  \cite[Section III.10]{Con}. 
\begin{lemma}\label{dualssubs}
Let $\mathscr{N}$ be a closed subspace of a Banach space $\mathscr{A}$.  Then 
$(\mathscr{A}/\mathscr{N})^*$ is isometrically isomorphic to $\mathscr{N}^{\perp}$, and
$\mathscr{N}^*$ is isometrically isomorphic to $\mathscr{A}^*/\mathscr{N}^{\perp}$.
\end{lemma}

The above lemma enables us to transform the extremal problem
\eqref{best_app} into an equivalent dual problem. 

\begin{theorem}\label{best_app_dual}
Suppose that $\mathscr{A}$ and $\mathscr{B}$ are real Banach spaces and $\mathbf{b}_0\in\mathscr{B}$. Let $H$ be a bounded linear operator from $\mathscr{A}$ to $\mathscr{B}$ and $H^*$ its adjoint operator. Then there holds 
\begin{equation}\label{dual_problem_formula}
   \inf\big\{ \|\mathbf{b}_0 - H\mathbf{a} \|_{\mathscr{B}} :\ \mathbf{a} \in \mathscr{A}  \big\} 
   =\sup\Bigg\{|\langle \mathbf{b}_0, \lambda \rangle_{\mathscr{B}}|:\ \|\lambda\|_{\mathscr{B}^*}\leq1,\ \lambda \in \kernel H^*\Bigg\}.
\end{equation}
\end{theorem} 

\begin{proof}
Let $\mathscr{M}$ be the subspace of $\mathscr{B}$ defined by \eqref{subspace}. It follows that  
\begin{equation}\label{norm-quotient}
    \inf \big\{ \|\mathbf{b}_0- H\mathbf{a} \|_{\mathscr{B}}:\mathbf{a} \in \mathscr{A} \} = \|\mathbf{b}_0 + \mathscr{M}\|_{\mathscr{B}/\mathscr{M}}.
\end{equation}
Lemma \ref{dualssubs} ensures that $$
\|\mathbf{b}_0 + \mathscr{M}\|_{\mathscr{B}/\mathscr{M}}=\sup\Bigg\{|\langle \mathbf{b}_0, \lambda \rangle_{\mathscr{B}}|:\ \|\lambda\|_{\mathscr{B}^*}\leq1,\ \lambda \in \mathscr{M}^{\perp} \Bigg\}.
$$
Substituting the above equation into equation \eqref{norm-quotient}, we get that
\begin{equation*}
    \inf \big\{ \|\mathbf{b}_0- H\mathbf{a} \|_{\mathscr{B}}:\mathbf{a} \in \mathscr{A} \} = \sup\Bigg\{|\langle \mathbf{b}_0, \lambda \rangle_{\mathscr{B}}|:\ \|\lambda\|_{\mathscr{B}^*}\leq1,\ \lambda \in \mathscr{M}^{\perp} \Bigg\}.
\end{equation*}
It suffices to identify $\mathscr{M}^{\perp}$ with the kernel of the adjoint operator $H^*$. Note that $\lambda \in \mathscr{M}^{\perp}$ is equivalent to $\langle H\mathbf{a}, \lambda \rangle_{\mathscr{B}} = 0$, for all $\mathbf{a}\in \mathscr{A}$. It follows from $\langle H\mathbf{a}, \lambda \rangle_{\mathscr{B}}=\langle \mathbf{a}, H^*\lambda \rangle_{\mathscr{A}},\ \mbox{for all}\ \mathbf{a}\in \mathscr{A},\ \lambda \in\mathscr{B}^*$, 
that $\lambda \in \mathscr{M}^{\perp}$ if and only if $\langle \mathbf{a}, H^*\lambda \rangle_{\mathscr{A}}=0$ for all $\mathbf{a}\in \mathscr{A}$. The latter is equivalent to $H^*\lambda = 0$, which is in turn equivalent to $\lambda \in \kernel H^*$. Therefore, the desired duality formula of the theorem holds.
\end{proof}

The duality result stated in Theorem \ref{best_app_dual} can even be extended to the cases when $\mathscr{M}$ is replaced by certain convex sets \cite{Ubhaya1990, Weinstein1990, Weinstein1991}. This result can also be obtained by using a convex analytic approach. By using the Fenchel
duality theorem, the best approximation problem can be characterized in the next lemma \cite{Barbu-Precupanu}. 

\begin{lemma}\label{best_approximation}
Let $\mathscr{B}$ be a real Banach space. Suppose that $\mathscr{M}$ is a nonempty closed convex set in $\mathscr{B}$ and $\mathbf{b}_0\in\mathscr{B}\setminus\mathscr{M}$. Then there holds
$$
\inf\big\{ \|\mathbf{b}_0-\mathbf{b}\|_{\mathscr{B}}:\mathbf{b} \in \mathscr{M}\}
=\sup_{\|\lambda\|_{\mathscr{B}^*}\leq 1} \inf_{\mathbf{b}\in\mathscr{M}}\langle\mathbf{b}_0-\mathbf{b},\lambda\rangle_{\mathscr{B}}.
$$
\end{lemma}

We now apply Lemma \ref{best_approximation} to the extremal problem \eqref{best_app} and obtain its dual problem described in Theorem \ref{best_app_dual}. When the closed convex set $\mathscr{M}$ coincides with the closure of the range of the operator $H$ defined by \eqref{subspace}, 
we represent problem \eqref{best_app} as a best approximation problem
$$
\inf\big\{ \|\mathbf{b}_0 - H\mathbf{a} \|_{\mathscr{B}} :\ \mathbf{a} \in \mathscr{A}  \big\} =\inf\big\{ \|\mathbf{b}_0-\mathbf{b}\|_{\mathscr{B}}:\mathbf{b} \in \mathscr{M}\}.
$$
This together with Lemma \ref{best_approximation} ensures that 
\begin{equation}\label{best_approximation1}
    \inf\big\{ \|\mathbf{b}_0 - H\mathbf{a} \|_{\mathscr{B}} :\ \mathbf{a} \in \mathscr{A}  \big\} =\sup_{\|\lambda\|_{\mathscr{B}^*}\leq 1} \inf_{\mathbf{b}\in\mathscr{M}}\langle\mathbf{b}_0-\mathbf{b},\lambda\rangle_{\mathscr{B}}.
\end{equation}
We next rewrite the right hand side of equation \eqref{best_approximation1} as that of equation \eqref{dual_problem_formula}. 
It follows that 
\begin{equation}\label{best_approximation2}
\sup_{\|\lambda\|_{\mathscr{B}^*}\leq 1} \inf_{\mathbf{b}\in\mathscr{M}}\langle\mathbf{b}_0-\mathbf{b},\lambda\rangle_{\mathscr{B}}
=\sup_{\|\lambda\|_{\mathscr{B}^*}\leq 1} \left\{\langle\mathbf{b}_0,\lambda\rangle_{\mathscr{B}}+\inf_{\mathbf{b}\in\mathscr{M}}\langle-\mathbf{b},\lambda\rangle_{\mathscr{B}}\right\}.
\end{equation}
Since $\mathscr{M}$ is a closed subspace, we have that 
\begin{equation*}
\inf_{\mathbf{b}\in\mathscr{M}}\langle-\mathbf{b},\lambda\rangle_{\mathscr{B}}
=\left\{\begin{array}{cc}
0, &\lambda\in\mathscr{M}^{\perp},\\
-\infty, &\lambda\notin\mathscr{M}^{\perp}.
\end{array}
\right.
\end{equation*}
Substituting the above equation into equation \eqref{best_approximation2}, we obtain that 
\begin{equation*}
\sup_{\|\lambda\|_{\mathscr{B}^*}\leq 1} \inf_{\mathbf{b}\in\mathscr{M}}\langle\mathbf{b}_0-\mathbf{b},\lambda\rangle_{\mathscr{B}}
=\sup\left\{\langle\mathbf{b}_0,\lambda\rangle_{\mathscr{B}}:\|\lambda\|_{\mathscr{B}^*}\leq 1, \lambda\in\mathscr{M}^{\perp}\right\}.
\end{equation*}
Again substituting the above equation into equation \eqref{best_approximation1}, we have that 
$$
\inf\big\{ \|\mathbf{b}_0 - H\mathbf{a} \|_{\mathscr{B}} :\ \mathbf{a} \in \mathscr{A}  \big\} =\sup\left\{\langle\mathbf{b}_0,\lambda\rangle_{\mathscr{B}}:\|\lambda\|_{\mathscr{B}^*}\leq1, \lambda\in\mathscr{M}^{\perp}\right\},
$$
which together with $\mathscr{M}^{\perp}=\kernel H^*$ leads to equation \eqref{dual_problem_formula}.


We now turn to characterizing the relation between the solutions of the extremal problem
\eqref{best_app} and its dual problem by utilizing the notion of a {\it norming functional.}  Recall that if $\mathbf{a} \in \mathscr{A}$ is a nonzero vector, a norming functional $\lambda$ of $\mathbf{a}$ is an element of $\mathscr{A}^*$  satisfying
$\|\lambda\|_{\mathscr{A}^*} =1$ and $\langle \mathbf{a}, \lambda \rangle_{\mathscr{A}} = \|\mathbf{a}\|_{\mathscr{A}}.$ The existence of such $\lambda$ is ensured by the Hahn-Banach theorem, though generally it need not be unique.  
Similarly, if $\lambda \in \mathscr{A}^*$, and $\mathbf{a} \in \mathscr{A}$ is a unit vector satisfying $\langle \mathbf{a}, \lambda\rangle_{\mathscr{A}} = \|\lambda\|_{\mathscr{A}^*}$, then we say that $\mathbf{a}$ is norming for $\lambda$. Such a vector $\mathbf{a}$ always exists if the space $\mathscr{A}$ is reflexive; more generally, we can always find a sequence of vectors that are approximately norming. The norming functionals are related to the duality mapping  $\Phi_{\mathscr{A}}$ from $\mathscr{A}$ to the collection of all subsets in $\mathscr{A}^*$, defined for all $\mathbf{a}\in\mathscr{A}$ by
$\Phi_{\mathscr{A}}(\mathbf{a}):=\{\lambda\in\mathscr{A}^*: \|\lambda\|_{\mathscr{A}^*}=\|\mathbf{a}\|_{\mathscr{A}}, \langle\mathbf{a},\lambda\rangle_{\mathscr{A}}=\|\lambda\|_{\mathscr{A}^*}
\|\mathbf{a}\|_{\mathscr{A}}\}.$
The norming functionals of $\mathbf{a}\in\mathscr{A}$, multiplied by $\|\mathbf{a}\|_{\mathscr{A}}$, are precisely the elements of the set $\Phi_{\mathscr{A}}(\mathbf{a})$.

The solutions of the original extremal problem \eqref{best_app} and its dual problem described in Theorem \ref{best_app_dual} are closely related as follows.
\begin{proposition}\label{norming_functional_general}
Suppose that $\mathscr{A}$ and $\mathscr{B}$ are real Banach spaces and $\mathbf{b}_0\in\mathscr{B}$. Let $H$ be a bounded linear operator from $\mathscr{A}$ to $\mathscr{B}$ and $H^*$ its adjoint operator. If $\hat{\mathbf{a}} \in \mathscr{A}$ has the property that 
\begin{equation}\label{valofinf0}
\|\mathbf{b}_0-H\hat{\mathbf{a}}\|_{\mathscr{B}}=\inf \big\{ \|\mathbf{b}_0 - H\mathbf{a} \|_{\mathscr{B}} :\ \mathbf{a} \in \mathscr{A}  \big\}>0,
\end{equation}
and $\hat{\lambda} \in \kernel H^*$ satisfies $\|\hat{\lambda}\|_{\mathscr{B}^*}=1$ and the condition
\begin{equation}\label{valofinf}
\langle \mathbf{b}_0, \hat{\lambda}\rangle_{\mathscr{B}} =\sup \Bigg\{ |\langle \mathbf{b}_0, \lambda\rangle_{\mathscr{B}}|:\|\lambda\|_{\mathscr{B}^*}\leq1, \ \lambda \in \kernel  H^* \Bigg\}>0,
\end{equation}  
then $(\mathbf{b}_0 - H\hat{\mathbf{a}})/\|\mathbf{b}_0 - H\hat{\mathbf{a}}\|_{\mathscr{B}}$ is norming for $\hat{\lambda}$ and
$\hat{\lambda}$ is norming for $\mathbf{b}_0 - H\hat{\mathbf{a}}$.
\end{proposition}
\begin{proof}
It is clear that the proposed norming vectors are of unit norm. It follows that 
$$
\langle\mathbf{b}_0 - H\hat{\mathbf{a}}, \hat{\lambda} \rangle_{\mathscr{B}} 
=\langle\mathbf{b}_0, \hat{\lambda} \rangle_{\mathscr{B}} - \langle H\hat{\mathbf{a}}, \hat{\lambda} \rangle_{\mathscr{B}} =\langle\mathbf{b}_0, \hat{\lambda} \rangle_{\mathscr{B}} - \langle \hat{\mathbf{a}}, H^* \hat{\lambda} \rangle_{\mathscr{A}}.
$$
Noting that $\hat\lambda \in \kernel H^*$, we observe from the above equation that
$\langle\mathbf{b}_0 - H\hat{\mathbf{a}}, \hat{\lambda} \rangle_{\mathscr{B}} 
=\langle\mathbf{b}_0, \hat{\lambda} \rangle_{\mathscr{B}},$
which, together with equation \eqref{valofinf}, further leads to 
$$
\langle\mathbf{b}_0 - H\hat{\mathbf{a}}, \hat{\lambda} \rangle_{\mathscr{B}} 
=\sup \{ |\langle \mathbf{b}_0, \lambda\rangle_{\mathscr{B}}|:\|\lambda\|_{\mathscr{B}^*}\leq1, \ \lambda \in \kernel  H^* \}.
$$
Substituting equations  \eqref{dual_problem_formula} and \eqref{valofinf0} into the above equation, we get that 
$\langle\mathbf{b}_0 - H\hat{\mathbf{a}}, \hat{\lambda} \rangle_{\mathscr{B}} 
=\|\mathbf{b}_0-H\hat{\mathbf{a}}\|_{\mathscr{B}}.$
It is obvious that   $\hat{\lambda}$ is norming for $\mathbf{b}_0 - H\hat{\mathbf{a}}$. Moreover, the conclusion that $(\mathbf{b}_0 - H\hat{\mathbf{a}})/\|\mathbf{b}_0 - H\hat{\mathbf{a}}\|_{\mathscr{B}}$ is norming for $\hat{\lambda}$ can be obtained by dividing both sides of the above equation by $\|\mathbf{b}_0-H\hat{\mathbf{a}}\|_{\mathscr{B}}$ and noting that $\|\hat{\lambda}\|_{\mathscr{B}^*}=1$. 
\end{proof}


In the rest of this section, we consider the special case that the Banach spaces $\mathscr{A}$ and $\mathscr{B}$ have pre-dual spaces. A normed space $\mathscr{A}_{*}$ is called a pre-dual space of a Banach space $\mathscr{A}$ if  $(\mathscr{A}_{*})^{*}=\mathscr{A}.$ Since the natural map is the isometrically imbedding map from $\mathscr{A}_*$ into $\mathscr{A}^*$, any element in $\mathscr{A}_*$ can be viewed as a bounded linear functional on $\mathscr{A}$, that is, an element in $\mathscr{A}^*$ and there holds 
$\langle \ell,\mathbf{a}\rangle_{\mathscr{A}_{*}}=\langle \mathbf{a},\ell\rangle_{\mathscr{A}}$, for all $\ell\in\mathscr{A}_*$ and all $\mathbf{a}\in\mathscr{A}$. The pre-dual space $\mathscr{A}_{*}$ guarantees that the Banach space $\mathscr{A}$ enjoys the weak$^{*}$ topology. The weak$^*$ topology of $\mathscr{A}$ is the smallest topology for $\mathscr{A}$ such that, for each $\ell\in\mathscr{A}_{*}$, the linear functional $\mathbf{a}\rightarrow \langle \mathbf{a},\ell\rangle_{\mathscr{A}}$ on $\mathscr{A}$ is continuous with respect to the topology. 
For a subset $\mathcal{N}$ of $\mathscr{A}$, we denote by $\overline{\mathcal{N}}^{w^*}$ the closure of $\mathcal{N}$ in the weak$^*$ topology of $\mathscr{A}$. We also give the name $\mathscr{N}_{\perp}$ to the subspace of $\mathscr{A}_*$ 
$$          
\mathscr{N}_{\perp} :=  \big\{ \ell \in  \mathscr{A}_*:\  \langle \mathbf{a},\ell  \rangle_{\mathscr{A}} = 0, \ \mbox{for all}\ \mathbf{a}\in \mathscr{N} \big\}.
$$

In the special case that the Banach spaces $\mathscr{A}$ and $\mathscr{B}$ have pre-duals $\mathscr{A}_*$ and $\mathscr{B}_*$, respectively, a similar result to Theorem \ref{best_app_dual} holds. 

\begin{theorem}\label{best_app_dual_pre_dual}
Suppose that $\mathscr{A}$ and $\mathscr{B}$ are real Banach spaces having the respective pre-dual spaces $\mathscr{A}_*$ and $\mathscr{B}_*$, $\mathbf{b}_0 \in \mathscr{B}$. Let $H$ be the adjoint operator of a  bounded linear operator $H_*$ from $\mathscr{B}_*$ to $\mathscr{A}_*$. If $\mathscr{M}$ defined by \eqref{subspace} satisfies $\overline{\mathscr{M}}^{w^*} =\mathscr{M}$, then there holds 
\begin{equation}\label{dual_problem_formula_pre_dual}
 \inf \big\{ \|\mathbf{b}_0 - H\mathbf{a} \|_{\mathscr{B}} :\ \mathbf{a} \in \mathscr{A} \big\} = \sup \Bigg\{ |\langle \mathbf{b}_0, \ell \rangle_{\mathscr{B}}| :\ \|\ell\|_{\mathscr{B}_*}\leq1,\ \ell \in \kernel H_* \Bigg\}.
\end{equation}
\end{theorem}
\begin{proof}
Again, associated with the closed subspace $\mathscr{M}$ of $\mathscr{B}$ defined by \eqref{subspace}, there holds equation \eqref{norm-quotient}. We apply the second part of Lemma \ref{dualssubs} to the right hand side of equation \eqref{norm-quotient}. To this end, we need to show that 
$(\kernel H_*)^{\perp}=  \mathscr{M}.$ 
We first identify  $\mathscr{M}_{\perp}$ with the kernel of the operator $H_*$. Specifically, by the definition of $\mathscr{M}_{\perp}$, we have that $\ell \in \mathscr{M}_{\perp}$ if and only if $\langle  H\mathbf{a}, \ell\rangle_{\mathscr{B}} = 0$, for all $\mathbf{a}\in \mathscr{A}$. The latter is equivalent to 
$\langle \mathbf{a}, H_*\ell \rangle_{\mathscr{A}}=0,\ \mbox{for all}\ \mathbf{a}\in \mathscr{A}.$ That is, $\ell \in \kernel H_*$. Thus, we conclude that $\mathscr{M}_{\perp}=\kernel H_*$, which further leads to $(\kernel H_*)^{\perp}=(\mathscr{M}_{\bot})^{\bot}$. It follows from Proposition 2.6.6 of \cite{Meg} that $(\mathscr{M}_{\bot})^{\bot}=\overline{\mathscr{M}}^{w*}$. We then get that 
$(\kernel H_*)^{\perp}= \overline{\mathscr{M}}^{w*}$. This together with the assumption that $\overline{\mathscr{M}}^{w^*} =\mathscr{M}$ yields that $(\kernel H_*)^{\perp}=\mathscr{M}$.
By employing Lemma \ref{dualssubs}, we get that
$\|\mathbf{b}_0 + \mathscr{M}\|_{\mathscr{B}/\mathscr{M}}=\sup \{ |\langle \mathbf{b}_0,\ell \rangle_{\mathscr{B}}| :\ \|\ell\|_{\mathscr{B}_*}\leq1,\ \ell \in \kernel H_* \}.$
Substituting the above equation into equation \eqref{norm-quotient}, we obtain the desired equation \eqref{dual_problem_formula_pre_dual}.
\end{proof}

We note that when $\mathscr{B}$ is an finite-dimensional Banach space, the assumption $\overline{\mathscr{M}}^{w^*} =\mathscr{M}$ holds. 
The added advantage in this particular situation is that the pre-dual of a space $\mathscr{A}$ is often a simpler or smaller space than its dual. For example, if $\mathscr{A} = \ell^1(\mathbb{N})$, the space of absolutely summable sequences, then its dual is $\mathscr{A}^* = \ell^{\infty}(\mathbb{N})$; however, its pre-dual is $\mathscr{A}_* = c_0(\mathbb{N})$, the space of sequences convergent to zero. 

For the special case when $\mathscr{A}$ and $\mathscr{B}$ have pre-duals, we can also characterize the relation between the solutions of the original problem and its dual problem. The proof of the following proposition is similar to that of Proposition \ref{norming_functional_general} and thus is omitted.

\begin{proposition}\label{norming_functional_predual}
Suppose that the assumptions in Theorem \ref{best_app_dual_pre_dual} hold. If $\hat{\mathbf{a}} \in \mathscr{A}$ has the property that 
\begin{equation*}\label{valofinf0_predual}
\|\mathbf{b}_0-H\hat{\mathbf{a}}\|_{\mathscr{B}}=\inf \big\{ \|\mathbf{b}_0 - H\mathbf{a} \|_{\mathscr{B}} :\ \mathbf{a} \in \mathscr{A}  \big\}>0,
\end{equation*}
and $\hat{\ell} \in \kernel H_*$ satisfies $\|\hat{\ell}\|_{\mathscr{B}_*}=1$ and the condition
\begin{equation*}\label{valofinf_predual}
\langle \mathbf{b}_0, \hat{\ell} \rangle_{\mathscr{B}} =\sup \Big\{ |\langle  \mathbf{b}_0,  \ell\rangle_{\mathscr{B}}|: \|\ell\|_{\mathscr{B}_*}\leq1, \ \ell \in \kernel H_*\Big\}>0,
\end{equation*}   
then $(\mathbf{b}_0 - H\hat{\mathbf{a}})/\|\mathbf{b}_0 - H\hat{\mathbf{a}}\|_{\mathscr{B}}$ is norming for $\hat{\ell}$, and
$\hat{\ell}$ is norming for $\mathbf{b}_0 - H\hat{\mathbf{a}}$.
\end{proposition}

Another advantage of the case in which the pre-duals $\mathscr{A}_*$ and $\mathscr{B}_*$
exist is that the extreme solution is generally 
attained. This is a result of the following statement.  

\begin{proposition}\label{extvalatt}
If $\mathscr{N}$ is a subspace of $\mathscr{A}_*$, and $\mathbf{a} \in \mathscr{A}$ with $\mathbf{a} \notin \mathscr{N}^{\perp}$, then there exists $\mathbf{a}' \in \mathscr{N}^{\perp}$ such that
$\|\mathbf{a} - \mathbf{a}' \|_{\mathscr{A}} = \dist(\mathbf{a},\mathscr{N}^{\perp}).$
\end{proposition}

As ever, see Appendix for a proof of the case $\mathscr{A}$ is separable. A quick consequence of Proposition \ref{extvalatt} is that the extreme value in \eqref{best_app} is attained. 

\begin{coro}\label{existence}
    Under the conditions of Theorem \ref{best_app_dual_pre_dual}, there exists $\hat{\mathbf{a}} \in \mathscr{A}$ such that the infimum in \eqref{dual_problem_formula_pre_dual} is attained.
\end{coro}

\section{The Regularized Extremal Problem }\label{dirsumsec}

The objective of this section is to develop a duality approach for solving the regularized extremal problem  \eqref{regprob}. The expression
$\big(  \|\mathbf{y}_0 - A\mathbf{x}\|_{\mathscr{Y}}^p + \rho\|B\mathbf{x}\|_{\mathscr{Z}}^p \big)^{1/p}$,
which appears in the regularized extremal problem \eqref{regprob}, is a monotone function of the loss function $\|\mathbf{y}_0 - A\mathbf{x}\|_{\mathscr{Y}}^p + \rho\|B\mathbf{x}\|_{\mathscr{Z}}^p$ and itself defines a norm on a direct sum space constructed from $\mathscr{Y}$ and $\mathscr{Z}$. Motivated by this observation, we reformulate the regularized extremal problem \eqref{regprob} as in \eqref{best_app} and apply a duality argument to transform  \eqref{regprob} into an equivalent dual problem. In practice, the dual problem, containing only finitely many free parameters, can be solvable by numerical methods. We then leverage the dual solution into a solution of the original problem.
Throughout the rest of this paper, we
always assume that the regularized extremal problem  \eqref{regprob} has a solution without further
mention. In particular, it is guaranteed by Corollary \ref{existence} that this assumption holds when
$\mathscr{X},\mathscr{Y}$ and $\mathscr{Z}$ have pre-dual spaces. 

We start with recalling the notion of the direct sum of two Banach spaces. The direct sum $\mathscr{A}\oplus\mathscr{B}$ of two Banach spaces $\mathscr{A}$ and $\mathscr{B}$ is the Cartesian product $\mathscr{A}\times\mathscr{B}$, made into a vector space via componentwise addition
$(\mathbf{a}_1,\mathbf{b}_1) + (\mathbf{a}_2,\mathbf{b}_2) := (\mathbf{a}_1 + \mathbf{a}_2, \mathbf{b}_1+\mathbf{b}_2)$ and scalar multiplication
$c(\mathbf{a}_1,\mathbf{b}_1) = (c\mathbf{a}_1,c\mathbf{b}_1).$
There are numerous ways to place a norm on a direct sum. We lay out a family of norms that could be placed on the direct sum, depending on the parameter $p\in[1,+\infty]$.  The associated dual spaces are also identified. For $p\in [1,+\infty)$, we denote by $\mathscr{A} \oplus_p\mathscr{B}$ the direct sum of $\mathscr{A}$ and $\mathscr{B}$ endowed with the norm
$
\|(\mathbf{a},\mathbf{b})\|_{\mathscr{A} \oplus_p \mathscr{B}} :=  \big(\|\mathbf{a}\|^p_{\mathscr{A}} +  \|\mathbf{b}\|^p_{\mathscr{B}} \big)^{1/p}.
$
We also denote by $\mathscr{A} \oplus_{\infty}\mathscr{B}$ the direct sum of $\mathscr{A}$ and $\mathscr{B}$ endowed with the norm
$
\|(\mathbf{a},\mathbf{b})\|_{\mathscr{A} \oplus_{\infty}\mathscr{B}} :=\max\big\{\|\mathbf{a}\|_{\mathscr{A}},
\|\mathbf{b}\|_{\mathscr{B}}\big\}.
$
It is easy to see that the spaces $\mathscr{A}\oplus_{p} \mathscr{B}$, $p\in[1,+\infty]$, are all Banach spaces. 
Let $p'$ be the H\"{o}lder conjugate of $p$ satisfying  $1/p+1/p'=1$. Then the dual space of $\mathscr{A}\oplus_{p} \mathscr{B}$ is isometrically isomorphic to $\mathscr{A}^*\oplus_{p'} \mathscr{B}^*$. 

We next identify the norming functionals for a nonzero element of the direct sum space $\mathscr{A} \oplus_p \mathscr{B}$,  $p\in(1,+\infty]$. Some function-theoretic preliminaries follow below, with the proofs being supplied in the appendix.
We first consider the case that $p,p'\in(1,+\infty)$. 
\begin{proposition}\label{normfcnldirsum}
   Let $\mathscr{A},\mathscr{B}$ be real Banach spaces and $p,p'\in(1,+\infty)$ satisfy $1/p+1/{p'} = 1$. Suppose that $(\lambda,\mu)\in \mathscr{A}^* \oplus_{p'}\mathscr{B}^*\setminus\{(0,0)\}$ is norming for $(\mathbf{a},\mathbf{b})\in \mathscr{A} \oplus_{p} \mathscr{B}\setminus\{(0,0)\}$. Then the following statements hold.
    
    \begin{enumerate}
        \item If $\mathbf{b}=0$,
        then $\mu=0$ and $\lambda$ is norming for $\mathbf{a}$;
        \item If $\mathbf{a}=0$, then $\lambda=0$ and $\mu$ is norming for $\mathbf{b}$;
        \item If $\mathbf{a}$ and $\mathbf{b}$ are both nonzero vectors, then $\lambda\neq 0$, $\mu\neq 0$, and   $\big(1+\|\mathbf{b}\|^p_{\mathscr{B}}/\|\mathbf{a}\|_{\mathscr{A}}^p\big)^{1/{p'}}\lambda$ is norming for $\mathbf{a}$
        and $\big(1+\|\mathbf{a}\|^p_{\mathscr{B}}/\|\mathbf{b}\|_{\mathscr{A}}^p\big)^{1/{p'}}\mu$ is norming for $\mathbf{b}$.
   \end{enumerate}

\end{proposition}

When $p=\infty$, ${p'}=1$, the norming functionals for a nonzero element of the direct sum space are characterized as follows.

\begin{proposition}\label{normfcnldirsum_p=infty}
    Let $\mathscr{A}$ and $\mathscr{B}$ be real Banach spaces. Suppose that $(\lambda,\mu)\in \mathscr{A}^* \oplus_{1} \mathscr{B}^*\setminus\{(0,0)\}$ is norming for $(\mathbf{a},\mathbf{b})\in \mathscr{A} \oplus_{\infty} \mathscr{B}\setminus\{(0,0)\}$. Then the following statements hold.
    
    \begin{enumerate}
        \item If $\|\mathbf{a}\|_{\mathscr{A}} > \|\mathbf{b}\|_{\mathscr{B}}$,
        then $\mu=0$ and $\lambda$ is norming for $\mathbf{a}$;
        \item If $\|\mathbf{a}\|_{\mathscr{A}} < \|\mathbf{b}\|_{\mathscr{B}}$, then $\lambda=0$ and $\mu$ is norming for $\mathbf{b}$;
        \item If $\|\mathbf{a}\|_{\mathscr{A}} = \|\mathbf{b}\|_{\mathscr{B}}$, then $\lambda,\mu$ satisfy one of the following three conditions: (i) $\mu=0$ and $\lambda$ is norming for $\mathbf{a}$; (ii) $\lambda=0$ and $\mu$ is norming for $\mathbf{b}$; (iii) $\lambda\neq 0$, $\mu\neq 0$, $\lambda/\|\lambda\|_{\mathscr{A}^*}$ is norming for $\mathbf{a}$, and $\mu/\|\mu\|_{\mathscr{B}^*}$ is norming for $\mathbf{b}$.
   \end{enumerate}
\end{proposition}

We now turn to rewriting  \eqref{regprob} as in  \eqref{best_app}. We  choose $\mathscr{A}:=\mathscr{X}$ and $\mathscr{B}:=\mathscr{Y}\oplus_{p}\mathscr{Z}$ for $p\in [1,+\infty)$. Associated with the bounded linear operators $A$, $B$ and the regularization parameter $\rho$, we define an operator $H:\mathscr{X}\rightarrow\mathscr{Y}\oplus_{p}\mathscr{Z}$ by \begin{equation}\label{operator_H}
    Hx:=(Ax,\rho^{1/p}Bx),\ \mbox{for all}\ x\in\mathscr{X}.
\end{equation} 
It is obvious that $H$ is a bounded linear operator on $\mathscr{X}$. The following lemma gives the adjoint operator of the operator $H$. 

\begin{lemma}\label{adjoint_H}
Suppose that $\mathscr{X},\mathscr{Y}$ and $\mathscr{Z}$ are real Banach spaces with the respective dual spaces $\mathscr{X}^*,\mathscr{Y}^*$ and $\mathscr{Z}^*$,  $A:\mathscr{X}\rightarrow\mathscr{Y}$, $B:\mathscr{X}\rightarrow\mathscr{Z}$ are bounded linear operators with the respective adjoint operators $A^*,B^*$. Let $\rho>0$ and $p\in[1,+\infty)$, $p'\in(1,+\infty]$ satisfy  $\textstyle{\frac{1}{p}}+\textstyle{\frac{1}{p'}} = 1$. If the operator $H$ is defined by \eqref{operator_H}, then the adjoint operator $H^*$ of $H$ has the form
\begin{equation*}
H^*(\lambda,\mu):=A^*\lambda+\rho^{1/p}B^*\mu,\ \ \mbox{for all}\ \ (\lambda,\mu)\in\mathscr{Y}^*\oplus_{p'}\mathscr{Z}^*.
\end{equation*}
\end{lemma}
\begin{proof}
We first note that
\begin{equation}\label{Expression-of-H*}
    \langle\mathbf{x}, H^*(\lambda,\mu)\rangle_{\mathscr{X}}=\langle H\mathbf{x}, (\lambda,\mu)\rangle_{\mathscr{Y}\oplus_{p}\mathscr{Z}},\  \mbox{for all}\  \mathbf{x}\in \mathscr{X}, (\lambda,\mu)\in\mathscr{Y}^*\oplus_{p'}\mathscr{Z}^*.
\end{equation}
By the definition \eqref{operator_H} of $H$, we have that
$\langle H\mathbf{x}, (\lambda,\mu)\rangle_{\mathscr{Y}\oplus_{p}\mathscr{Z}} 
=\langle A\mathbf{x},\lambda\rangle_{\mathscr{Y}}+ \rho^{1/p}\langle B\mathbf{x},\mu\rangle_{\mathscr{Z}},$
which further leads to 
$$
\langle H\mathbf{x}, (\lambda,\mu)\rangle_{\mathscr{Y}\oplus_{p}\mathscr{Z}} =\langle \mathbf{x}, A^*\lambda\rangle_{\mathscr{X}}+\rho^{1/p}\langle \mathbf{x}, B^*\mu\rangle_{\mathscr{X}}
=\langle \mathbf{x}, A^*\lambda+\rho^{1/p}B^*\mu\rangle_{\mathscr{X}}.
$$
Substituting the above equation into the right hand side of \eqref{Expression-of-H*},
we obtain for all $\mathbf{x}\in \mathscr{X}$ and all  $(\lambda,\mu)\in\mathscr{Y}^*\oplus_{p'}\mathscr{Z}^*$ that 
$\langle\mathbf{x}, H^*(\lambda,\mu)\rangle_{\mathscr{X}}=\langle \mathbf{x}, A^*\lambda+\rho^{1/p}B^*\mu\rangle_{\mathscr{X}},$
which leads to the desired representation of $H^*$.
\end{proof}

By using the direct sum $\mathscr{Y}\oplus_p \mathscr{Z}$, $p\in[1,+\infty)$, and the operator $H$ defined by \eqref{operator_H}, we rewrite the objective function of the regularized extremal problem \eqref{regprob} as 
$$
 (\| {\mathbf{y}_0} - A{\mathbf{x}} \|^p_{\mathscr{Y}} + \rho\|B\mathbf{x}\|^p_{\mathscr{Z}})^{1/p} =\| ({\mathbf{y}_0} - A{\mathbf{x}}, -\rho^{1/p}B\mathbf{x})\|_{\mathscr{Y}\oplus_p\mathscr{Z}}=\|(\mathbf{y}_0,0)-H\mathbf{x}\|_{\mathscr{Y}\oplus_p\mathscr{Z}}.
$$
Thus, we rewrite the regularized extremal problem \eqref{regprob} as 
\begin{equation}\label{rewritten_regularization}
\inf\left\{\|(\mathbf{y}_0,0)-H\mathbf{x}\|_{\mathscr{Y}\oplus_p\mathscr{Z}}: \mathbf{x} \in \mathscr{X} \right\}.
\end{equation}
Applying Theorem \ref{best_app_dual} to problem \eqref{rewritten_regularization}, we get the dual problem of \eqref{regprob}.
\begin{theorem}\label{maindualthm}
Suppose that $\mathscr{X},\mathscr{Y}$ and $\mathscr{Z}$ are real
Banach spaces with the respective dual spaces $\mathscr{X}^*,\mathscr{Y}^*$ and $\mathscr{Z}^*$,  $A:\mathscr{X}\rightarrow\mathscr{Y}$, $B:\mathscr{X}\rightarrow\mathscr{Z}$ are bounded linear operators with the adjoint operators $A^*,B^*$, respectively. Let 
$\mathbf{y}_0\in\mathscr{Y}$ and $\rho>0$. Then there holds
\begin{align}
   &\inf\big\{ \|\mathbf{y}_0-A\mathbf{x}\|_{\mathscr{Y}}+\rho \|B\mathbf{x}\|_{\mathscr{Z}} :\ \mathbf{x} \in \mathscr{X} \big\} \nonumber\\ 
   =&
   \sup\Bigg\{|\langle \mathbf{y}_0,\lambda\rangle_{\mathscr{Y}}|:\max\{\|\lambda\|_{\mathscr{Y}^*}, \| \mu\|_{\mathscr{Z}^*}\}\leq1,\ A^*\lambda+\rho B^*\mu=0\Bigg\}.\label{extremalprobs_p=1}
\end{align}
If $p,p'\in(1,+\infty)$ such that $1/p+1/{p'} = 1$, then there holds 
\begin{align}
    &\inf \big\{ \big(\|\mathbf{y}_0-A\mathbf{x}\|^p_{\mathscr{Y}}+\rho \|B\mathbf{x}\|^p_{\mathscr{Z}}\big)^{1/p} :\ \mathbf{x} \in \mathscr{X} \big\}\nonumber\\ 
    =& \sup\Bigg\{|\langle \mathbf{y}_0,\lambda\rangle_{\mathscr{Y}}|:\ \big(\|\lambda\|^{p'}_{\mathscr{Y}^*}+\| \mu\|_{\mathscr{Z}^*}^{p'}\big)^{1/{p'}}\leq1,\  A^*\lambda+\rho^{1/p}B^*\mu=0 \Bigg\}.\label{extremalprobs_p>1}
\end{align}
\end{theorem}

\begin{proof} 
The regularized extremal problem \eqref{regprob} with $p\in[1,+\infty)$ can be reformulated as in
\eqref{rewritten_regularization}, which has the form \eqref{best_app} with $\mathscr{A}$, $\mathscr{B}$ and $\mathbf{b}_0$ being replaced by $\mathscr{X}$, $\mathscr{Y}\oplus_p \mathscr{Z}$ and $(\mathbf{y}_0,0)$, respectively, and $H$ being defined by \eqref{operator_H}. 
We then apply Theorem \ref{best_app_dual} to problem \eqref{rewritten_regularization} and obtain that 
\begin{equation}\label{apply_22_32}
    \inf\left\{\|(\mathbf{y}_0,0)-H\mathbf{x}\|_{\mathscr{Y}\oplus_p\mathscr{Z}}: \mathbf{x} \in \mathscr{X} \right\}
    =\sup\Bigg\{|\langle \mathbf{y}_0, \lambda\rangle_{\mathscr{Y}}|:\|(\lambda,\mu)\|_{\mathscr{Y}^*\oplus_{p'}\mathscr{Z}^*}\leq1,\ (\lambda,\mu) \in \kernel H^* \Bigg\}.
\end{equation}
Lemma \ref{adjoint_H} ensures that $(\lambda,\mu) \in \kernel H^*$ if and only if $ \lambda\in\mathscr{Y}^*$, $\mu\in\mathscr{Z}^*$ satisfy 
$A^*\lambda+\rho^{1/p}B^*\mu=0.$
Substituting the above equation into equation \eqref{apply_22_32}, noting that problems \eqref{regprob} and
\eqref{rewritten_regularization} are equivalent, we get that 
\begin{align*}
   & \inf \big\{ \big(\|\mathbf{y}_0-A\mathbf{x}\|^p_{\mathscr{Y}}+\rho \|B\mathbf{x}\|^p_{\mathscr{Z}}\big)^{1/p} :\ \mathbf{x} \in \mathscr{X} \big\}\nonumber\\ 
   =&\sup\Bigg\{|\langle \mathbf{y}_0, \lambda\rangle_{\mathscr{Y}}|:\|(\lambda,\mu)\|_{\mathscr{Y}^*\oplus_{p'}\mathscr{Z}^*}\leq1,\   A^*\lambda+\rho^{1/p}B^*\mu=0 \Bigg\}.
 \end{align*}
When $p=1$ and $p'=+\infty$, substituting the norm $\|\cdot\|_{\mathscr{Y}^*\oplus_{\infty}\mathscr{Z}^*}$ into the above equation leads to equation \eqref{extremalprobs_p=1}. 
Likewise, when $p,p'\in(1,+\infty)$, substituting the norm $\|\cdot\|_{\mathscr{Y}^*\oplus_{p'}\mathscr{Z}^*}$ into the above equation leads to equation \eqref{extremalprobs_p>1}.  
\end{proof}

In practice, such as in the machine learning problem \eqref{genregprob0}, the space $\mathscr{Y}$ is finite dimensional.  As a result, the dual extremal problem (that is, the suprema in \eqref{extremalprobs_p=1} or \eqref{extremalprobs_p>1}) has only finitely many parameters, namely, the components of the vector $\lambda\in\mathscr{Y}^*$.  The  dual problem which is of {\it finite dimension} can therefore be solved using numerical methods. In this way, the duality argument offers a useful reduction. 
We describe this practical case as follows. Let $\mathscr{Y}=\mathbb{R}^m$ endowed with a norm $\|\cdot\|_{\mathbb{R}^m}$ and $\mathbf{y}_0=(y_j:j\in\mathbb{N}_m)\in\mathbb{R}^m$. We denote by $\|\cdot\|^*_{\mathbb{R}^m}$ the dual norm of $\|\cdot\|_{\mathbb{R}^m}$. Suppose that $\mathscr{X}$ is a real Banach space with the dual space $\mathscr{X}^*$ and $\mathbf{a}_j$, $j\in\mathbb{N}_m,$ are a finite number of linearly independent elements of $\mathscr{X}^*$ and the operator $A:\mathscr{X}\rightarrow\mathbb{R}^m$ is defined by 
$A\mathbf{x}:=\left(\langle\mathbf{x},\mathbf{a}_j\rangle_{\mathscr{X}}:j\in\mathbb{N}_m\right)$, for all $\mathbf{x}\in\mathscr{X}$.
It follows from the definition of $A$ that for all $\mathbf{x}\in\mathscr{X}$ and all $\lambda:=(\lambda_j:j\in\mathbb{N}_m)\in\mathbb{R}^m$,
$
\langle \mathbf{x},A^*\lambda\rangle_{\mathscr{X}}
=\sum_{j\in\mathbb{N}_m}\lambda_j\langle\mathbf{x},\mathbf{a}_j\rangle_{\mathscr{X}}
=\left\langle\mathbf{x},\sum_{j\in\mathbb{N}_m}\lambda_j\mathbf{a}_j\right\rangle_{\mathscr{X}},
$
which leads to
\begin{equation}\label{adjoint-operator}
A^*\lambda=\sum_{j\in\mathbb{N}_m}\lambda_j\mathbf{a}_j,\ \mbox{for all}\ \lambda:=(\lambda_j:j\in\mathbb{N}_m)\in\mathbb{R}^m.
\end{equation}
We also let $\mathscr{Z}$ be a real Banach space with the dual space  $\mathscr{Z}^*$. Suppose that $B:\mathscr{X}\rightarrow\mathscr{Z}$ is a bounded linear operator satisfying $\range B=\mathscr{Z}$, that is, $B^*$ is injective. It follows from equation \eqref{adjoint-operator} and $A^*\lambda+\rho^{1/p} B^*\mu=0$ that 
\begin{equation}\label{representer_theorem_B}
\mu=-\rho^{-1/p}\sum_{j\in\mathbb{N}_m}\lambda_j(B^*)^{-1}\mathbf{a}_j.
\end{equation}
Consequently, we conclude by Theorem \ref{maindualthm} that 
\begin{align}
    &\inf \big\{ \big(\|\mathbf{y}_0-A\mathbf{x}\|_{\mathbb{R}^m}^p+\rho \|B\mathbf{x}\|^p_{\mathscr{Z}}\big)^{1/p} :\ \mathbf{x} \in \mathscr{X} \big\}\nonumber\\ 
    =& \sup\Bigg\{\sum_{j\in\mathbb{N}_m}y_j\lambda_j: \Big\|\Big((\lambda_j:j\in\mathbb{N}_m),-\rho^{-1/p}\sum_{j\in\mathbb{N}_m}\lambda_j(B^*)^{-1}\mathbf{a}_j\Big)\Bigg\|_{\mathbb{R}^m\oplus_{p'}\mathscr{Z}^*}\leq1\Bigg\}.\label{extremalprobs_p>1_finite}
\end{align}
The dual extremal problem in
\eqref{extremalprobs_p>1_finite}
has only a finite number of real parameters, namely the $m$ components of $\lambda$. Notice that the expression
$\big\|\big((\lambda_j:j\in\mathbb{N}_m),-\rho^{-1/p}\sum_{j\in\mathbb{N}_m}\lambda_j(B^*)^{-1}\mathbf{a}_j\big)\big\|_{\mathbb{R}^m\oplus_{p'}\mathscr{Z}^*}$ defines a norm on $\mathbb{R}^m$, which must be equivalent to the Euclidean norm.  
Consequently,
the dual problem is to maximize a linear function of $m$ variables over some compact set in $\mathbb{R}^m$. The maximum must be attained, though not necessarily uniquely. The fact that the solution $(\hat\lambda,\hat\mu)$ to the dual problem \eqref{extremalprobs_p>1_finite} takes the form 
\eqref{representer_theorem_B}
can be viewed as a {\it representer theorem}. 
The original representer theorem \cite{scholkopf2001} for a learning method was in the setting of a Hilbert space, which is self-dual. 
Representer theorems for solutions of regularization problems in Banach
spaces have received considerable attention in the literature. A systematic study of this topic was conducted in \cite{WX2020}. The resulting representer theorem states that the
solution lies in a subdifferential set of the norm function evaluated at a finite linear combination of given functionals. The regularization problem in an infinite dimensional Banach space was then reduced to a finite dimensional optimization problem of the coefficients of the linear combination. The dual problem \eqref{extremalprobs_p>1_finite} coincides with the resulting finite dimensional optimization problem in the representer theorem.

By combining Propositions \ref{norming_functional_general} and \ref{normfcnldirsum}, we can characterize the solution of the regularized extremal problem  \eqref{regprob} with $p\in(1,+\infty)$ by using the norming functionals of the solution of the dual problem stated in Theorem \ref{maindualthm}.

\begin{proposition}\label{norming_functional_general_regularization_p>1}
Suppose that $\mathscr{X},\mathscr{Y}$ and $\mathscr{Z}$ are real Banach spaces with the respective dual spaces $\mathscr{X}^*,\mathscr{Y}^*$ and $\mathscr{Z}^*$,  $A:\mathscr{X}\rightarrow\mathscr{Y}$, $B:\mathscr{X}\rightarrow\mathscr{Z}$ are bounded linear operators with the adjoint operators $A^*,B^*$, respectively. Let 
$\mathbf{y}_0\in\mathscr{Y}$, $\rho>0$ and $p,p'\in(1,+\infty)$ satisfy $1/p+1/p'=1$. If $\hat{\mathbf{x}} \in \mathscr{X}$ has the property that 
\begin{equation}\label{valofinf0_regularization_p>1}
\big(\|\mathbf{y}_0-A\hat{\mathbf{x}}\|^p_{\mathscr{Y}}+\rho \|B\hat{\mathbf{x}}\|^p_{\mathscr{Z}}\big)^{1/p}=\inf \big\{ \big(\|\mathbf{y}_0-A\mathbf{x}\|^p_{\mathscr{Y}}+\rho \|B\mathbf{x}\|^p_{\mathscr{Z}}\big)^{1/p} :\ \mathbf{x} \in \mathscr{X} \big\}>0,
\end{equation} and $(\hat{\lambda},\hat{\mu}) \in \mathscr{Y}^*\oplus_{p'}\mathscr{Z}^*$ satisfies $\big(\|\hat\lambda\|^{p'}_{\mathscr{Y}^*}+\| \hat\mu\|_{\mathscr{Z}^*}^{p'}\big)^{1/{p'}}=1$,  $A^*\hat\lambda+\rho^{1/p}B^*\hat\mu=0$ and 
\begin{equation*}\label{valofinf_regularization_p>1}
\langle \mathbf{y}_0, \hat\lambda\rangle_{\mathscr{Y}} =\sup\Bigg\{|\langle \mathbf{y}_0, \lambda\rangle_{\mathscr{Y}}|:\big(\|\lambda\|^{p'}_{\mathscr{Y}^*}+\| \mu\|_{\mathscr{Z}^*}^{p'}\big)^{1/{p'}}\leq1,\  A^*\lambda+\rho^{1/p}B^*\mu=0 \Bigg\}>0,
\end{equation*} 
then the following statements hold.
 \begin{enumerate}
        \item If $\hat\mu=0$,
        then $B\hat{\mathbf{x}}=0$ and $\|\mathbf{y}_0 - A\hat{\mathbf{x}}\|_{\mathscr{Y}}=\langle \mathbf{y}_0, \hat\lambda\rangle_{\mathscr{Y}} $;
        
        
        \item If $\hat\lambda$ and $\hat\mu$ are both nonzero, then $\mathbf{y_0}-A\hat{\mathbf{x}}\neq 0$, $B\hat{\mathbf{x}}\neq 0$, and $\alpha(\mathbf{y_0}-A\hat{\mathbf{x}})$ is norming for $\hat\lambda$ and $\beta B\hat{\mathbf{x}}$ is norming for $\hat\mu$, where  $\alpha:=\big(1+\|\hat\mu\|^{p'}_{\mathscr{Z}^*}/\|\hat\lambda\|_{\mathscr{Y}^*}^{p'}\big)^{1/{p}}/\langle \mathbf{y}_0, \hat\lambda\rangle_{\mathscr{Y}}$ and  $\beta:=-\rho^{1/p}\big(1+\|\hat\lambda\|^{p'}_{\mathscr{Y}^*}/\|\hat\mu\|_{\mathscr{Z}^*}^{p'}\big)^{1/{p}}/\langle \mathbf{y}_0, \hat\lambda\rangle_{\mathscr{Y}}$.
       \end{enumerate}
\end{proposition}

\begin{proof}
As pointed out earlier, the regularized extremal problem \eqref{regprob} with $p\in(1,+\infty)$ can be reformulated as in \eqref{best_app} with $\mathscr{A}$, $\mathscr{B}$ and $\mathbf{b}_0$ being replaced by $\mathscr{X}$, $\mathscr{Y}\oplus_p \mathscr{Z}$ and $(\mathbf{y}_0,0)$, respectively, and $H$ being defined by \eqref{operator_H}. We conclude by Proposition  \ref{norming_functional_general} that $(\mathbf{y}_0 - A\hat{\mathbf{x}}, -\rho^{1/p} B\hat{\mathbf{x}})/\|(\mathbf{y}_0 - A\hat{\mathbf{x}}, -\rho^{1/p} B\hat{\mathbf{x}})\|_{\mathscr{Y}\oplus_{p}\mathscr{Z}}$ is norming for $(\hat{\lambda},\hat{\mu})\in\mathscr{Y}^*\oplus_{p'}\mathscr{Z}^*$. According to statement 1 of Proposition \ref{normfcnldirsum}, we have that if $\hat\mu=0$, then $B\hat{\mathbf{x}}=0$. This, together with  \begin{equation}\label{inf=sup1}
(\|\mathbf{y}_0-A\hat{\mathbf{x}}\|^p_{\mathscr{Y}}+\rho \|B\hat{\mathbf{x}}\|^p_{\mathscr{Z}})^{1/p}=\langle \mathbf{y}_0, \hat\lambda\rangle_{\mathscr{Y}},
\end{equation}
leads to $\|\mathbf{y}_0-A\hat{\mathbf{x}}\|_{\mathscr{Y}}=\langle \mathbf{y}_0, \hat\lambda\rangle_{\mathscr{Y}}$. 
For the case that $\hat\lambda, \hat\mu$ are both nonzero, we get by statement 3 of Proposition \ref{normfcnldirsum} that $\mathbf{y_0}-A\hat{\mathbf{x}}\neq 0$, $B\hat{\mathbf{x}}\neq 0$.  
Moreover, $\big(1+\|\hat\mu\|^{p'}_{\mathscr{Z}^*}/\|\hat\lambda\|_{\mathscr{Y}^*}^{p'}\big)^{1/{p}}(\mathbf{y_0}-A\hat{\mathbf{x}})/\|(\mathbf{y}_0 - A\hat{\mathbf{x}}, -\rho^{1/p} B\hat{\mathbf{x}})\|_{\mathscr{Y}\oplus_{p}\mathscr{Z}}$ is norming for $\hat\lambda$ and $\big(1+\|\hat\lambda\|^{p'}_{\mathscr{Z}^*}/\|\hat\mu\|_{\mathscr{Y}^*}^{p'}\big)^{1/{p}}(-\rho^{1/p} B\hat{\mathbf{x}})/\|(\mathbf{y}_0 - A\hat{\mathbf{x}}, -\rho^{1/p} B\hat{\mathbf{x}})\|_{\mathscr{Y}\oplus_{p}\mathscr{Z}}$ is norming for $\hat\mu$. By setting $\alpha$, $\beta$ as in this proposition and equation \eqref{inf=sup1}, we get the desired result.
\end{proof}

The next proposition concerns the characterization of the solution of the regularized extremal problem  \eqref{regprob} with $p=1$, which can be obtained by employing Propositions \ref{norming_functional_general} and \ref{normfcnldirsum_p=infty}. 
\begin{proposition}\label{norming_functional_general_regularization_p=1}
Suppose that $\mathscr{X},\mathscr{Y}$ and $\mathscr{Z}$ are real Banach spaces with the respective dual spaces $\mathscr{X}^*,\mathscr{Y}^*$ and $\mathscr{Z}^*$,  $A:\mathscr{X}\rightarrow\mathscr{Y}$, $B:\mathscr{X}\rightarrow\mathscr{Z}$ are bounded linear operators with the adjoint operators $A^*,B^*$, respectively. Let 
$\mathbf{y}_0\in\mathscr{Y}$ and $\rho>0$. If $\hat{\mathbf{x}} \in \mathscr{X}$ has the property that 
\begin{equation}\label{valofinf0_regularization_p=1}
\|\mathbf{y}_0-A\hat{\mathbf{x}}\|_{\mathscr{Y}}+\rho \|B\hat{\mathbf{x}}\|_{\mathscr{Z}}=\inf \big\{ \|\mathbf{y}_0-A\mathbf{x}\|_{\mathscr{Y}}+\rho \|B\mathbf{x}\|_{\mathscr{Z}} :\ \mathbf{x} \in \mathscr{X} \big\}>0,
\end{equation} and $(\hat{\lambda},\hat{\mu}) \in \mathscr{Y}^*\oplus_{\infty}\mathscr{Z}^*$ satisfies $\max\{\|\hat\lambda\|_{\mathscr{Y}^*}, \| \hat\mu\|_{\mathscr{Z}^*}\}=1$,  $A^*\hat\lambda+\rho B^*\hat\mu=0$ and 
\begin{equation*}\label{valofinf_regularization_p=1}
\langle \mathbf{y}_0, \hat\lambda\rangle_{\mathscr{Y}} =\sup\Bigg\{|\langle \mathbf{y}_0, \lambda\rangle_{\mathscr{Y}}|:\max\{\|\lambda\|_{\mathscr{Y}^*}, \| \mu\|_{\mathscr{Z}^*}\}\leq1,\  A^*\lambda+\rho B^*\mu=0 \Bigg\}>0,
\end{equation*}   
then the following statements hold.

  \begin{enumerate}
     \item If $\|\hat\lambda\|_{\mathscr{Y}^*} > \| \hat\mu\|_{\mathscr{Z}^*}$, then $B\hat{\mathbf{x}} = \mathbf{0}$ and $\|\mathbf{y}_0 - A\hat{\mathbf{x}}\|_{\mathscr{Y}}=\langle \mathbf{y}_0, \hat\lambda\rangle_{\mathscr{Y}} $;

     \item If $\|\hat\lambda\|_{\mathscr{Y}^*} < \| \hat\mu\|_{\mathscr{Z}^*}$, then $A\hat{\mathbf{x}}=\mathbf{y}_0$ and $\rho\|B\hat{\mathbf{x}}\|_{\mathscr{Z}}=\langle \mathbf{y}_0, \hat\lambda\rangle_{\mathscr{Y}} $;

     \item If $\|\hat\lambda\|_{\mathscr{Y}^*} = \| \hat\mu\|_{\mathscr{Z}^*}$, then $\hat{\mathbf{x}}$ satisfies one of the following three conditions: (i) $B\hat{\mathbf{x}} = \mathbf{0}$ and $\|\mathbf{y}_0 - A\hat{\mathbf{x}}\|_{\mathscr{Y}}=\langle \mathbf{y}_0, \hat\lambda\rangle_{\mathscr{Y}} $; (ii) $A\hat{\mathbf{x}}=\mathbf{y}_0$ and $\rho\|B\hat{\mathbf{x}}\|_{\mathscr{Z}}=\langle \mathbf{y}_0, \hat\lambda\rangle_{\mathscr{Y}} $; (iii)  $\mathbf{y_0}-A\hat{\mathbf{x}}\neq 0$, $B\hat{\mathbf{x}}\neq 0$,  $(\mathbf{y_0}-A\hat{\mathbf{x}})/\|\mathbf{y_0}-A\hat{\mathbf{x}}\|_{\mathscr{Y}}$ is norming for $\hat\lambda$ and $- B\hat{\mathbf{x}}/\|B\hat{\mathbf{x}}\|_{\mathscr{Z}}$ is norming for $\hat\mu$.
  \end{enumerate}
\end{proposition}

\begin{proof}
Note that the regularized extremal problem \eqref{regprob} can be reformulated as in \eqref{best_app} with $\mathscr{A}$, $\mathscr{B}$ and $\mathbf{b}_0$ being replaced by $\mathscr{X}$, $\mathscr{Y}\oplus_p \mathscr{Z}$ and $(\mathbf{y}_0,0)$, respectively, and $H$ being defined by \eqref{operator_H}. Proposition  \ref{norming_functional_general} ensures that $(\mathbf{y}_0 - A\hat{\mathbf{x}}, -\rho B\hat{\mathbf{x}})/\|(\mathbf{y}_0 - A\hat{\mathbf{x}}, -\rho B\hat{\mathbf{x}})\|_{\mathscr{Y}\oplus_{1}\mathscr{Z}}$ is norming for $(\hat{\lambda},\hat{\mu})\in\mathscr{Y}^*\oplus_{\infty}\mathscr{Z}^*$.
We characterize the norming vector by Proposition \ref{normfcnldirsum_p=infty} with $\mathscr{A}:=\mathscr{Y}^*$ and $\mathscr{B}:=\mathscr{Z}^*$. If $\|\hat\lambda\|_{\mathscr{Y}^*} > \| \hat\mu\|_{\mathscr{Z}^*}$, statement 1 of Proposition \ref{normfcnldirsum_p=infty} leads directly to $B\hat{\mathbf{x}}=0$. Moreover, Theorem \ref{maindualthm} guarantees that \begin{equation}\label{inf=sup}
\|\mathbf{y}_0-A\hat{\mathbf{x}}\|_{\mathscr{Y}}+\rho \|B\hat{\mathbf{x}}\|_{\mathscr{Z}}=\langle \mathbf{y}_0, \hat\lambda\rangle_{\mathscr{Y}},
\end{equation}
which together with $B\hat{\mathbf{x}}=0$ leads to $\|\mathbf{y}_0-A\hat{\mathbf{x}}\|_{\mathscr{Y}}=\langle \mathbf{y}_0, \hat\lambda\rangle_{\mathscr{Y}}$. If $\|\hat\lambda\|_{\mathscr{Y}^*} < \| \hat\mu\|_{\mathscr{Z}^*}$, we get by statement 2 of Proposition \ref{normfcnldirsum_p=infty} that $A\hat{\mathbf{x}}=\mathbf{y}_0$. Again by equation \eqref{inf=sup}, we obtain that $\rho\|B\hat{\mathbf{x}}\|_{\mathscr{Z}}=\langle \mathbf{y}_0, \hat\lambda\rangle_{\mathscr{Y}}$. It suffices to consider the case that $\|\hat\lambda\|_{\mathscr{Y}^*} = \| \hat\mu\|_{\mathscr{Z}^*}$. In this case, Statement 3 of Proposition \ref{normfcnldirsum_p=infty} shows that the norming vector  $(\mathbf{y}_0 - A\hat{\mathbf{x}}, -\rho B\hat{\mathbf{x}})/\|(\mathbf{y}_0 - A\hat{\mathbf{x}}, -\rho B\hat{\mathbf{x}})\|_{\mathscr{Y}\oplus_{1}\mathscr{Z}}$ may satisfy one of three conditions. By similar arguments as above, the first one leads to $B\hat{\mathbf{x}} = \mathbf{0}$ and $\|\mathbf{y}_0 - A\hat{\mathbf{x}}\|_{\mathscr{Y}}=\langle \mathbf{y}_0, \hat\lambda\rangle_{\mathscr{Y}} $ and the second one ensures that $A\hat{\mathbf{x}}=\mathbf{y}_0$ and $\rho\|B\hat{\mathbf{x}}\|_{\mathscr{Z}}=\langle \mathbf{y}_0, \hat\lambda\rangle_{\mathscr{Y}}$. In addition, the third condition coincides with  $\mathbf{y_0}-A\hat{\mathbf{x}}\neq 0$, $B\hat{\mathbf{x}}\neq 0$,  $(\mathbf{y_0}-A\hat{\mathbf{x}})/\|\mathbf{y_0}-A\hat{\mathbf{x}}\|_{\mathscr{Y}}$ is norming for $\hat\lambda$ and $- B\hat{\mathbf{x}}/\|B\hat{\mathbf{x}}\|_{\mathscr{Z}}$ is norming for $\hat\mu$. This completes the proof of this proposition.
\end{proof}

In the special case that $\mathscr{X}$, $\mathscr{Y}$ and $\mathscr{Z}$ have pre-dual spaces, there is a similar dual problem. 

\begin{theorem}\label{mainpredualthm}  
Suppose that $\mathscr{X},\mathscr{Y}$ and $\mathscr{Z}$ are real
Banach spaces with the respective pre-dual spaces $\mathscr{X}_*,\mathscr{Y}_*$ and $\mathscr{Z}_*$,   $A_*:\mathscr{Y}_*\rightarrow\mathscr{X}_*$, $B_*:\mathscr{Z}_*\rightarrow\mathscr{X}_*$ are bounded linear operators with the adjoint operators $A,B$, respectively. Let 
$\mathbf{y}_0\in\mathscr{Y}$ and $\rho>0$. If $\mathscr{Y}$ is finite dimensional and $\range B=\mathscr{Z}$, then there holds 
\begin{align}\label{extremalprobs3x}
 &\inf\big\{ \|\mathbf{y}_0-A\mathbf{x}\|_{\mathscr{Y}}+\rho \|B\mathbf{x}\|_{\mathscr{Z}} :\ \mathbf{x} \in \mathscr{X} \big\} \nonumber\\ 
   =&
   \sup\Bigg\{|\langle\mathbf{y}_0, \lambda\rangle_{\mathscr{Y}}|:\ \max\{\|\lambda\|_{\mathscr{Y}_*}, \|\mu\|_{\mathscr{Z}_*}\}\leq1,\ A_*\lambda+\rho B_*\mu=0\Bigg\}.
\end{align}
If $p,p'\in(1,+\infty)$ such that $1/p+1/{p'} = 1$, then there holds 
\begin{align}\label{extremalprobs4x}
     &\inf \big\{ \big(\|\mathbf{y}_0-A\mathbf{x}\|^p_{\mathscr{Y}}+\rho \|B\mathbf{x}\|^p_{\mathscr{Z}}\big)^{1/p} :\ \mathbf{x} \in \mathscr{X} \big\}\nonumber\\ 
    =& \sup\Bigg\{|\langle\mathbf{y}_0,\lambda\rangle_{\mathscr{Y}}|:\big(\|\lambda\|^{p'}_{\mathscr{Y}_*}+\| \mu\|_{\mathscr{Z}_*}^{p'}\big)^{1/{p'}}\leq1,\ A_*\lambda+\rho^{1/p}B_*\mu=0 \Bigg\}.
\end{align}
\end{theorem}
\begin{proof}
As pointed out earlier, the regularized extremal problem \eqref{regprob} can be rewritten as in \eqref{best_app} with $\mathscr{A}$, $\mathscr{B}$ and $\mathbf{b}_0$ being replaced by $\mathscr{X}$, $\mathscr{Y}\oplus_p \mathscr{Z}$ and $(\mathbf{y}_0,0)$, respectively, and $H$ being defined by \eqref{operator_H}. We prove this theorem by using Theorem \ref{best_app_dual_pre_dual}. Note that $\mathscr{Y}\oplus_p \mathscr{Z}$ has the pre-dual space $\mathscr{Y}_*\oplus_{p'} \mathscr{Z}_*$ and $H$ is the adjoint operator of  $H_*:\mathscr{Y}_*\oplus_{p'} \mathscr{Z}_*\rightarrow\mathscr{X}_*$ defined by 
$H_*(\lambda,\mu):=A_*\lambda+\rho^{1/p}B_*\mu,$ for all  $(\lambda,\mu)\in\mathscr{Y}_*\oplus_{p'} \mathscr{Z}_*.$
By setting $\mathscr{M}:=\overline{\{ (Ax,\rho^{1/p}Bx): \ x\in\mathscr{X}\}},$
we have that 
$$
\overline{\mathscr{M}}^{w^*} =\overline{\{Ax: \ x\in\mathscr{X}\}}^{w^*}\times\overline{\{\rho^{1/p}Bx: \ x\in\mathscr{X}\}}^{w^*},
$$
which, together with the assumptions that $\mathscr{Y}$ is finite dimensional and $\range B=\mathscr{Z}$, yields that $\overline{\mathscr{M}}^{w^*}=\mathscr{M}$. That is, the hypotheses of Theorem \ref{best_app_dual_pre_dual} are satisfied.  Hence, Theorem \ref{best_app_dual_pre_dual} ensures that
\begin{align*}\label{simple_to_regularization1_pre_dual}
   &\inf \big\{ \big(\|\mathbf{y}_0-A\mathbf{x}\|^p_{\mathscr{Y}}+\rho \|B\mathbf{x}\|^p_{\mathscr{Z}}\big)^{1/p} :\ \mathbf{x} \in \mathscr{X} \big\}\\
   =&\sup\Bigg\{|\langle \mathbf{y}_0,\lambda \rangle_{\mathscr{Y}}|:\|(\lambda,\mu)\|_{\mathscr{Y}_*\oplus_{p'}\mathscr{Z}_*}\leq1,\ (\lambda,\mu) \in \kernel H_* \Bigg\}.
\end{align*}
Substituting the kernel of $H_*$ and the definition of the norm of the direct sum $\mathscr{Y}_*\oplus_{p'}\mathscr{Z}_*$, $p'\in(1,+\infty]$, into the above equation leads to 
the desired equations \eqref{extremalprobs3x} and \eqref{extremalprobs4x}. 
\end{proof}

By specializing Proposition \ref{norming_functional_predual} to the regularized extremal problem \eqref{regprob} when $\mathscr{X}$, $\mathscr{Y}$ and $\mathscr{Z}$ have pre-dual spaces, we relate the solution of \eqref{regprob} to the dual solution. The proofs of the following two propositions are similar to those of Propositions \ref{norming_functional_general_regularization_p>1} and \ref{norming_functional_general_regularization_p=1} and thus are omitted. 

\begin{proposition}\label{norming_functional_predual_regularization_p>1}
Suppose that $\mathscr{X},\mathscr{Y}$ and $\mathscr{Z}$ are real
Banach spaces with the respective pre-dual spaces $\mathscr{X}_*,\mathscr{Y}_*$ and $\mathscr{Z}_*$,   $A_*:\mathscr{Y}_*\rightarrow\mathscr{X}_*$, $B_*:\mathscr{Z}_*\rightarrow\mathscr{X}_*$ are bounded linear operators with the adjoint operators $A,B$, respectively.
In addition, suppose that $\mathscr{Y}$ is finite dimensional and $\range B=\mathscr{Z}$.   Let 
$\mathbf{y}_0\in\mathscr{Y}$, $\rho>0$ and $p,p'\in(1,+\infty)$, satisfy  $\textstyle{\frac{1}{p}}+\textstyle{\frac{1}{p'}} = 1$. If $\hat{\mathbf{x}} \in \mathscr{X}$ has the property \eqref{valofinf0_regularization_p>1} and $(\hat{\lambda},\hat{\mu}) \in \mathscr{Y}_*\oplus_{p'}\mathscr{Z}_*$ satisfies $\big(\|\hat\lambda\|^{p'}_{\mathscr{Y}_*}+\| \hat\mu\|_{\mathscr{Z}_*}^{p'}\big)^{1/{p'}}=1$, $A_*\hat\lambda+\rho^{1/p}B_*\hat\mu=0$ and 
\begin{equation*}\label{valofinf_regularization_predual}
\langle\mathbf{y}_0,\hat\lambda \rangle_{\mathscr{Y}} =\sup\Bigg\{|\langle \mathbf{y}_0,\lambda \rangle_{\mathscr{Y}}|:\big(\|\lambda\|^{p'}_{\mathscr{Y}_*}+\| \mu\|_{\mathscr{Z}_*}^{p'}\big)^{1/{p'}}\leq1,\  A_*\lambda+\rho^{1/p}B_*\mu=0
\Bigg\}>0,
\end{equation*}  
then the following statements hold.
 \begin{enumerate}
        \item If $\hat\mu=0$,
        then $B\hat{\mathbf{x}}=0$ and $\|\mathbf{y}_0 - A\hat{\mathbf{x}}\|_{\mathscr{Y}}=\langle\mathbf{y}_0,\hat\lambda \rangle_{\mathscr{Y}}$;
        
        
        \item If $\hat\lambda$ and $\hat\mu$ are both nonzero, then $\mathbf{y_0}-A\hat{\mathbf{x}}\neq 0$, $B\hat{\mathbf{x}}\neq 0$, and $\alpha(\mathbf{y_0}-A\hat{\mathbf{x}})$ is norming for $\hat\lambda$ and $\beta B\hat{\mathbf{x}}$ is norming for $\hat\mu$, where  $\alpha:=\big(1+\|\hat\mu\|^{p'}_{\mathscr{Z}_*}/\|\hat\lambda\|_{\mathscr{Y}_*}^{p'}\big)^{1/{p}}/\langle\mathbf{y}_0,\hat\lambda \rangle_{\mathscr{Y}}$ and  $\beta:=-\rho^{1/p}\big(1+\|\hat\lambda\|^{p'}_{\mathscr{Y}_*}/\|\hat\mu\|_{\mathscr{Z}_*}^{p'}\big)^{1/{p}}/\langle\mathbf{y}_0,\hat\lambda \rangle_{\mathscr{Y}}$.
       \end{enumerate}
\end{proposition}

\begin{proposition}\label{norming_functional_predual_regularization_p=1}
Suppose that $\mathscr{X},\mathscr{Y}$ and $\mathscr{Z}$ are real
Banach spaces with the respective pre-dual spaces $\mathscr{X}_*,\mathscr{Y}_*$ and $\mathscr{Z}_*$,   $A_*:\mathscr{Y}_*\rightarrow\mathscr{X}_*$, $B_*:\mathscr{Z}_*\rightarrow\mathscr{X}_*$ are bounded linear operators with the adjoint operators $A,B$, respectively.
In addition, suppose that $\mathscr{Y}$ is finite dimensional and $\range B=\mathscr{Z}$. Let 
$\mathbf{y}_0\in\mathscr{Y}$ and $\rho>0$. If $\hat{\mathbf{x}} \in \mathscr{X}$ has the property \eqref{valofinf0_regularization_p=1}
and $(\hat{\lambda},\hat{\mu}) \in \mathscr{Y}_*\oplus_{\infty}\mathscr{Z}_*$ satisfies $\max\{\|\hat\lambda\|_{\mathscr{Y}_*}, \| \hat\mu\|_{\mathscr{Z}_*}\}=1$,  $A_*\hat\lambda+\rho B_*\hat\mu=0$ and 
\begin{equation*}\label{valofinf_regularization}
\langle\mathbf{y}_0,\hat\lambda \rangle_{\mathscr{Y}} =\sup\Bigg\{|\langle  \mathbf{y}_0,\lambda \rangle_{\mathscr{Y}}|:\max\{\|\lambda\|_{\mathscr{Y}_*}, \| \mu\|_{\mathscr{Z}_*}\}\leq1,\  A_*\lambda+\rho B_*\mu=0 \Bigg\}>0,
\end{equation*}   
then the following statements hold.

  \begin{enumerate}
     \item If $\|\hat\lambda\|_{\mathscr{Y}_*} > \| \hat\mu\|_{\mathscr{Z}_*}$, then $B\hat{\mathbf{x}} = \mathbf{0}$ and $\|\mathbf{y}_0 - A\hat{\mathbf{x}}\|_{\mathscr{Y}}=\langle\mathbf{y}_0,\hat\lambda \rangle_{\mathscr{Y}} $;

     \item If $\|\hat\lambda\|_{\mathscr{Y}_*} < \| \hat\mu\|_{\mathscr{Z}_*}$, then $A\hat{\mathbf{x}}=\mathbf{y}_0$ and $\rho\|B\hat{\mathbf{x}}\|_{\mathscr{Z}}=\langle\mathbf{y}_0,\hat\lambda \rangle_{\mathscr{Y}} $;

     \item If $\|\hat\lambda\|_{\mathscr{Y}_*} = \| \hat\mu\|_{\mathscr{Z}_*}$, then $\hat{\mathbf{x}}$ satisfy one of the following three conditions: (i) $B\hat{\mathbf{x}} = \mathbf{0}$ and $\|\mathbf{y}_0 - A\hat{\mathbf{x}}\|_{\mathscr{Y}}=\langle\mathbf{y}_0,\hat\lambda \rangle_{\mathscr{Y}}$; (ii) $A\hat{\mathbf{x}}=\mathbf{y}_0$ and $\rho\|B\hat{\mathbf{x}}\|_{\mathscr{Z}}=\langle\mathbf{y}_0,\hat\lambda \rangle_{\mathscr{Y}} $; (iii)  $\mathbf{y_0}-A\hat{\mathbf{x}}\neq 0$, $B\hat{\mathbf{x}}\neq 0$,  $(\mathbf{y_0}-A\hat{\mathbf{x}})/\|\mathbf{y_0}-A\hat{\mathbf{x}}\|_{\mathscr{Y}}$ is norming for $\hat\lambda$ and $- B\hat{\mathbf{x}}/\|B\hat{\mathbf{x}}\|_{\mathscr{Z}}$ is norming for $\hat\mu$.
  \end{enumerate}
\end{proposition}


\section{Regularized Extremal Problem in $\ell_1(\mathbb{N})$}

In this section, we illustrate the duality approach developed in the previous section with the regularized extremal problem in $\ell_1(\mathbb{N})$.

We describe the regularized extremal problem in $\ell_1(\mathbb{N})$. Let $\mathscr{X} = \mathscr{Z} = \ell_1(\mathbb{N})$, the Banach space consisting of all real sequences $\mathbf{x}:=(x_j:j\in\mathbb{N})$ such that 
$\|\mathbf{x}\|_1:=\sum_{j\in\mathbb{N}}|x_j|<+\infty$. It is known that $\ell_1(\mathbb{N})$ has $c_0(\mathbb{N})$ as its pre-dual space, where $c_0(\mathbb{N})$ denotes the Banach space of all real sequences  $\mathbf{a}:=(a_j:j\in\mathbb{N})$ converging to $0$ as $j\to\infty$, endowed with $\|\mathbf{a}\|_{\infty}:=\sup\{|a_j|: j\in\mathbb{N}\}<+\infty$. 
For $m\in\mathbb{N}$, let $\mathscr{Y}=\mathbb{R}^m$ endowed with a norm $\|\cdot\|_{\mathbb{R}^m}$ and $\mathbf{y}_0=(y_j:j\in\mathbb{N}_m)\in\mathbb{R}^m$. For a norm $\|\cdot\|_{\mathbb{R}^m}$ on $\mathbb{R}^m$, we denote by $\|\cdot\|^*_{\mathbb{R}^m}$ its dual norm. Suppose that $\mathbf{a}_j:=(a_{j,k}:k\in\mathbb{N})$, $j\in\mathbb{N}_m,$ are a finite number of linearly independent elements of $c_0(\mathbb{N})$ and let $A$ be the semi-infinite matrix whose $m$ rows are  $\mathbf{a}_1,\mathbf{a}_2,\ldots,\mathbf{a}_m$.  It is easy to see that $A$ determines a bounded linear operator from $\ell_1(\mathbb{N})$ to $\mathbb{R}^m$.
We also choose the operator $B$ as the identity operator on $\ell_1(\mathbb{N})$. Let $p\in[1,+\infty)$ and $\rho>0$. We then consider the regularized extremal problem
\begin{equation}\label{specific_example}
\inf \big\{(\|\mathbf{y}_0-A\mathbf{x}\|^p_{\mathbb{R}^m}+\rho \|\mathbf{x}\|_{1}^p)^{1/p}:\ \mathbf{x} \in \ell_1(\mathbb{N}) \big\}.
\end{equation}

We now solve the regularized extremal problem \eqref{specific_example} 
by applying the duality approach developed in section 3. To this end, we first establish the dual problem of \eqref{specific_example} by using Theorem \ref{mainpredualthm}. It is easy to see that the operator $A$ is the adjoint operator of  $A_*:\mathbb{R}^m\rightarrow c_0(\mathbb{N})$ defined by 
\begin{equation}\label{A_*}
A_*\lambda:=\sum_{j\in\mathbb{N}_m}\lambda_j\mathbf{a}_j,\ \mbox{for all}\ \lambda:=(\lambda_j:j\in\mathbb{N}_m).
\end{equation}
Theorem \ref{mainpredualthm} ensures that the dual problem of problem \eqref{specific_example} has the form

\begin{equation}\label{specific_example_dual}
  \sup\left\{\sum_{j\in\mathbb{N}_m}y_j\lambda_j:\left\|\Bigg(\lambda,-\rho^{-1/p}\sum_{j\in\mathbb{N}_m}\lambda_j\mathbf{a}_j\Bigg)\right\|_{\mathbb{R}^m\oplus_{p'}c_0(\mathbb{N})}\leq1,\ \lambda:=(\lambda_j:j\in\mathbb{N}_m) \in \mathbb{R}^m \right\}. 
\end{equation}

We next show that the dual problem \eqref{specific_example_dual} is indeed a finite dimensional optimization problem. We first consider the case that $p=1$, $p'=+\infty$, in which the dual problem can be represented as
\begin{equation}\label{specific_example_dual_p=1}
    \sup\left\{\sum_{j\in\mathbb{N}_m}y_j\lambda_j : \|(\lambda_j:j\in\mathbb{N}_m)\|^*_{\mathbb{R}^m}\leq 1, \left\|\sum_{j\in\mathbb{N}_m}\lambda_j\mathbf{a}_j \right\|_{\infty}\leq \rho, \lambda_j \in \mathbb{R},j\in\mathbb{N}_m   \right\}.
\end{equation}
Although this optimization problem has only finitely many
parameters $\lambda_j$, $j\in\mathbb{N}_m$, a certain infinite dimensional aspect is hidden in the resulting finite dimensional problem. In fact, the constraint 
$\left\|\sum_{j\in\mathbb{N}_m}\lambda_j\mathbf{a}_j \right\|_{\infty}\leq \rho$ involves infinitely many constraints 
$\left|\sum_{j\in\mathbb{N}_m}\lambda_ja_{j,k}\right|\leq \rho, \ k\in\mathbb{N}.$
To overcome this obstacle, we define for each $k\in\mathbb{N}$
\begin{equation}\label{uk}
U_k:=\left\{\lambda:=(\lambda_j:j\in\mathbb{N}_m)\in\mathbb{R}^m:-\rho\leq\sum_{j\in\mathbb{N}_m}\lambda_ja_{j,k}\leq \rho\right\},
\end{equation}
and set $U:=\bigcap_{k\in\mathbb{N}} U_k$. It has been proved in \cite{CX1} that the set $U$ is the intersection of finitely many of the regions $U_k$. That is to say, the dual problem \eqref{specific_example_dual_p=1} is an optimization problem with finitely many constraints. Hence, we can obtain a solution of \eqref{specific_example_dual_p=1} by using standard numerical methods. 
Likewise, when $p,p'\in(1,+\infty)$, the dual problem \begin{equation}\label{specific_example_dual_p>1}
  \sup\left\{\sum_{j\in\mathbb{N}_m}y_j\lambda_j:( \|(\lambda_j:j\in\mathbb{N}_m)\|^*_{\mathbb{R}^m})^{p'}+\rho^{1-p'}\left\|\sum_{j\in\mathbb{N}_m}\lambda_j\mathbf{a}_j\right\|^{p'}_{\infty}\leq1,\  \lambda_j \in \mathbb{R},j\in\mathbb{N}_m \right\}
\end{equation}
also has infinitely many constraints 
$( \|(\lambda_j:j\in\mathbb{N}_m)\|^*_{\mathbb{R}^m})^{p'}+\rho^{1-p'}\left|\sum_{j\in\mathbb{N}_m}\lambda_ja_{j,k}\right|^{p'}\leq1, \ k\in\mathbb{N}.$
However, by similar arguments as in  \cite{CX1}, we can reduce the above constraints to finite number of constraints. Accordingly, the dual problem \eqref{specific_example_dual_p>1} can also be solved by numerical methods. 

Finally, we consider solving the original regularized extremal problem \eqref{specific_example}. According to the relation between the solutions of the regularized extremal problem \eqref{specific_example} and its dual problem \eqref{specific_example_dual}, we can obtain the solution of \eqref{specific_example} by solving an equivalent finite dimensional optimization problem. To see this, we introduce some notation. For each $\mathbf{c}:=(c_j:j\in\mathbb{N})\in c_0(\mathbb{N})$, we denote by $\mathbb{N}(\mathbf{c})$ the index set on which the sequence $\mathbf{c}$ achieves its supremum norm $\|\mathbf{c}\|_\infty$, that is,
$\mathbb{N}(\mathbf{c}):=\left\{j\in\mathbb{N}:|c_j|=\|\mathbf{c}\|_\infty\right\}$. It follows from  $\lim_{j\rightarrow+\infty}c_j=0$ that the cardinality of index set $\mathbb{N}(\mathbf{c})$, denoted by $n_{\mathbf{c}}$, is finite. It has been proved in \cite{CX1} that there holds for any norming functional $\mathbf{x}\in\ell_1(\mathbb{N})$ of $\mathbf{c}\in c_0(\mathbb{N})\setminus\{0\}$ that $\supp(\mathbf{x})\subseteq\mathbb{N}(\mathbf{c})$. Here, the support $\supp(\mathbf{x})$ of $\mathbf{x}$ is defined to be  the index set on which $\mathbf{x}$ is nonzero. For $\mathbf{c}\in c_0(\mathbb{N})$ with $\mathbb{N}(\mathbf{c}):=\{k_j\in\mathbb{N}:j\in\mathbb{N}_{n_{\mathbf{c}}}\}$, we truncate the semi-infinite matrix $A$ to obtain a matrix $A_{\mathbf{c}}:=[h_{ij}:i\in\mathbb{N}_m, j\in\mathbb{N}_{n_{\mathbf{c}}}]\in\mathbb{R}^{m\times n_{\mathbf{c}}}$ by 
\begin{equation}\label{trancate_matrix}
    h_{ij}:=a_{i,k_j}, \quad i\in\mathbb{N}_n,\ j\in\mathbb{N}_{n_{\mathbf{c}}}.
\end{equation}

\begin{proposition}\label{solution_specific_example2}
   Suppose that $\mathbf{a}_j$, $j\in\mathbb{N}_m,$ are a finite number of linearly independent elements of $c_0(\mathbb{N})$, $A$ is the infinite matrix whose $m$ rows are $\mathbf{a}_j$, $j\in\mathbb{N}_m,$ $\mathbf{y}_0=(y_j:j\in\mathbb{N}_m)\in\mathbb{R}^m$, $\rho>0$ and $p\in[1,+\infty)$. Let $\hat\lambda$ be a solution of the dual problem \eqref{specific_example_dual}, $\hat\mu:=A_*\hat\lambda$ and $\mathbb{N}(\hat\mu):=\{k_j\in\mathbb{N}:j\in\mathbb{N}_{n_{\hat\mu}}\}$. If $\hat{\mathbf{z}}:=(\hat{z}_j:j\in\mathbb{N}_{n_{\hat\mu}})$ is a solution of the optimization problem 
     \begin{equation}\label{specific_example_finite}
     \inf \big\{(\|\mathbf{y}_0-A_{\hat\mu}\mathbf{z}\|^p_{\mathbb{R}^m}+\rho \|\mathbf{z}\|^p_{1})^{1/p}:\ \mathbf{z} \in \mathbb{R}^{n_{\hat\mu}} \big\}, 
     \end{equation}
     then $\hat{\mathbf{x}}:=(\hat{x}_j:j\in\mathbb{N})$ with $\hat{x}_{k_j}:=\hat{z}_j$, $j\in\mathbb{N}_{n_{\hat\mu}}$, and $\hat{x}_j:=0$, $j\notin\mathbb{N}(\hat\mu)$, is a solution of \eqref{specific_example}.
     
\end{proposition}
\begin{proof}
By employing Propositions \ref{norming_functional_predual_regularization_p>1} and \ref{norming_functional_predual_regularization_p=1}, we have that for any solution $\hat{\mathbf{x}}$ of \eqref{specific_example}, there holds that $\hat{\mathbf{x}}=0$ or $-\hat{\mathbf{x}}/\|\hat{\mathbf{x}}\|_{\ell_1(\mathbb{N})}$ is norming for $-\rho^{-1/p}\hat\mu$. Note that $\supp(\hat{\mathbf{x}})\subseteq\mathbb{N}(\mathbf{\hat\mu})$ holds in both cases. 
Hence, we rewrite \eqref{specific_example} as an equivalent form \begin{equation}\label{specific_example_finite1}
\inf \big\{(\|\mathbf{y}_0-A\mathbf{x}\|^p_{\mathbb{R}^m}+\rho \|\mathbf{x}\|^p_{1})^{1/p}:\ \supp(\mathbf{x})\subseteq\mathbb{N}(\mathbf{\hat\mu}),\  \mathbf{x} \in \ell_1(\mathbb{N}) \big\}.
\end{equation}
By use of the matrix $A_{\hat\mu}$, the optimization problem \eqref{specific_example_finite1} can be further represented by \eqref{specific_example_finite}. Moreover, we can obtain a solution $\hat{\mathbf{x}}:=(\hat{x}_j:j\in\mathbb{N})$ of  \eqref{specific_example_finite1} through augmenting a solution $\hat{\mathbf{z}}:=(\hat{z}_j:j\in\mathbb{N}_{n_{\hat\mu}})$ of \eqref{specific_example_finite} by setting $\hat{x}_{k_j}:=\hat{z}_j$, $j\in\mathbb{N}_{n_{\hat\mu}}$, and $\hat{x}_j:=0$, $j\notin\mathbb{N}(\hat\mu)$.
\end{proof}

Based on Propositions \ref{norming_functional_predual_regularization_p=1} and  \ref{solution_specific_example2}, we develop a scheme for finding a solution of the infinite dimensional regularized extremal problem \eqref{specific_example} with $p=1$. 
\begin{itemize}
\setlength{\itemsep}{2pt}
    \item 1: Solve the dual problem \eqref{specific_example_dual_p=1} and obtain a solution  $\hat\lambda$. 
    
    \item 2: Compute $\rho\|\hat\lambda\|_{\mathbb{R}^m}$ and $\|A_*\hat\lambda\|_{\infty}$. 
    
    \item 3: If $\rho\|\hat\lambda\|_{\mathbb{R}^m}>\|A_*\hat\lambda\|_{\infty}$, obtain a solution of \eqref{specific_example} by $\hat{\mathbf{x}}:=0$. 
    
    \item 4: If $\rho\|\hat\lambda\|_{\mathbb{R}^m}\leq\|A_*\hat\lambda\|_{\infty}$, determine the index set $\mathbb{N}(\hat\mu):=\{k_j\in\mathbb{N}:j\in\mathbb{N}_{n_{\hat\mu}}\}$ and generate the matrix $A_{\hat\mu}$ with  $\hat\mu:=A_*\hat\lambda$. Solve the finite dimensional optimization problem 
    \eqref{specific_example_finite} and obtain a solution $\hat{\mathbf{z}}$.
    
    \item 5: Obtain a solution of \eqref{specific_example} by setting $\hat{\mathbf{x}}:=(\hat{x}_j:j\in\mathbb{N})$ with $\hat{x}_{k_j}:=\hat{z}_j$, $j\in\mathbb{N}_{n_{\hat\mu}}$, and $\hat{x}_j:=0$, $j\notin\mathbb{N}(\hat\mu)$.
\end{itemize}

Propositions \ref{norming_functional_predual_regularization_p>1} and  \ref{solution_specific_example2} may also provide a scheme for solving the regularized extremal problem \eqref{specific_example} with $p\in(1,+\infty)$. 

\begin{itemize}
\setlength{\itemsep}{2pt}
    \item 1: Solve the dual problem \eqref{specific_example_dual_p>1} and obtain a solution  $\hat\lambda$. 
    
    \item 2: Compute  $\hat\mu:=A_*\hat\lambda$. 
    
    \item 3: If $\hat\mu=0$, obtain a solution of \eqref{specific_example} by $\hat{\mathbf{x}}:=0$. 
    
    \item 4: If $\hat\mu\neq0$, determine the index set $\mathbb{N}(\hat\mu):=\{k_j\in\mathbb{N}:j\in\mathbb{N}_{n_{\hat\mu}}\}$ and generate the matrix $A_{\hat\mu}$. Solve the finite dimensional optimization problem 
    \eqref{specific_example_finite} and obtain a solution $\hat{\mathbf{z}}$.
    
    \item 5: Obtain a solution of \eqref{specific_example} by setting $\hat{\mathbf{x}}:=(\hat{x}_j:j\in\mathbb{N})$ with $\hat{x}_{k_j}:=\hat{z}_j$, $j\in\mathbb{N}_{n_{\hat\mu}}$, and $\hat{x}_j:=0$, $j\notin\mathbb{N}(\hat\mu)$.
\end{itemize}

\section{Numerical Experiments}
This section is devoted to the presentation of three numerical experiments which illustrate the feasibility of the dual approach developed in this paper. In the first experiment, we consider a small size problem to illustrate every key step of the proposed approach, and in the last two experiments, we consider  relatively large size problems to demonstrate the effectiveness of the method. All the experiments are performed with Matlab R2018a on an Intel Core I5 (8-core) with 1.80 GHz and 8 Gb RAM.

We consider solving the regularized extremal problem  
\begin{equation}\label{numerical}
\inf \big\{\|\mathbf{y}_0-A\mathbf{x}\|_{1}+\rho \|\mathbf{x}\|_{1}:\ \mathbf{x} \in \ell_1(\mathbb{N}) \big\}.
\end{equation} 
Note that the space $\ell_1(\mathbb{N})$ is of infinite dimension and thus, numerical methods used in compress-sensing or image processing are not directly applicable to problem \eqref{numerical}. We will employ the duality approach developed in this paper to solve problem \eqref{numerical}.
We first describe the choice of the data $\mathbf{y}_0$ and the semi-infinite matrix $A$ to be used in the experiments. Let $m$ be a positive even integer. We choose $A$ as a semi-infinite matrix whose rows $\mathbf{a}_j:=(a_{j,k}:k\in\mathbb{N})$, $j\in\mathbb{N}_m$, are $m$ linearly independent elements of $c_0(\mathbb{N})$, defined by  
\begin{equation*}\label{linear_independence_a}
a_{j,k}:=\frac{\cos(jk)}{k}, \  a_{j+m/2,k}:=\frac{\sin(jk)}{k},\ \mbox{for all}\ j\in\mathbb{N}_{m/2}\ \mbox{and all}\ k\in\mathbb{N}.
\end{equation*} 
We take $\mathbf{x}_0:=(x_k:k\in\mathbb{N})\in\ell_1(\mathbb{N})$ with $x_k=\frac{1}{10k^2}$, $k\in\mathbb{N}$ and generate the noise free data $\mathbf{y}:=A\mathbf{x}_0$. The noisy data is modeled as $\mathbf{y}_0:= \mathbf{y} + \eta$,
where $\eta$ is the Gaussian noise with the standard deviation $\sigma = 1.0 \times 10^{-3}(\mathrm{max} \mathbf{y} - \mathrm{min} \mathbf{y})$. 

Following the scheme described in section 4, we need to solve the dual problem 
\begin{equation}\label{numerical_dual}
    \sup\left\{\sum_{j\in\mathbb{N}_m}y_j\lambda_j : \|(\lambda_j:j\in\mathbb{N}_m)\|_{\infty}\leq 1, \left\|\sum_{j\in\mathbb{N}_m}\lambda_j\mathbf{a}_j \right\|_{\infty}\leq \rho
    \right\}.
\end{equation}
Note that the constraint $\left\|\sum_{j\in\mathbb{N}_m} \lambda_j\mathbf{a}_j \right\|_{\infty}\leq \rho$ of problem \eqref{numerical_dual} involves $m$ infinite dimensional 
vectors $\mathbf{a}_j$. It is equivalent to the constraints
$\left|\sum_{j\in\mathbb{N}_m}\lambda_ja_{j,k}\right|\leq \rho$, 
$k\in\mathbb{N}$, which defines a polytope in $\mathbb{R}^m$.  
To describe the polytope precisely, we define the sets $U_k$, $k\in\mathbb{N}$ as in \eqref{uk}. It is known \cite{CX1} that the set $U:=\bigcap_{k\in\mathbb{N}} U_k$ is the intersection of finitely many of the regions $U_k$, that is,  $U=\bigcap_{k\in\mathbb{N}_{n_0}} U_k$ for some $n_0\in\mathbb{N}$. We need to identify the number $n_0$. 
To this end, we propose two ideas. When $m$ is small, 
we set for each $n\in\mathbb{N}$,  
$\widetilde{U}_n:=\bigcap_{k\in\mathbb{N}_n} U_k$ and compute the vertices of the regions $\widetilde{U}_n$, $n\in\mathbb{N}$, in the increasing order of $n$. 
The computation stops at $n_0\in\mathbb{N}$ when $\widetilde{U}_{n_0}$ and $\widetilde{U}_{n_0+1}$ have the same vertices. Note that the number of the vertices  increases exponentially as the dimension $m$ of the space $\mathbb{R}^m$ increases. Hence, when $m$ is large, this idea is not practical.
In this case, for $l\in \mathbb{N}$ we may instead define
\begin{align*}
S(l):= \sup\left\{\sum_{j\in\mathbb{N}_m}y_j\lambda_j: |\lambda_j|\leq 1, j\in \mathbb{N}_m,  \left|\sum_{j\in\mathbb{N}_m}\lambda_ja_{j,k}\right|\leq \rho, k\in\mathbb{N}_l
    \right\}
\end{align*}
and empirically choose $n_0\in \mathbb{N}$ such that $S(n_0)=S(n_0+1)$. Computing $S(l)$ may be done by using linear programming software. In the following three numerical experiments we use these two ideas respectively.

In the first experiment, we solve problem \eqref{numerical} with $m=12$. In this case, we identify the number of the effective constraints in the dual problem by computing the vertices of the regions $\widetilde{U}_n$, $n\in\mathbb{N}$, in the increasing order of $n$.  Due to the large number of the vertices we do not report all of them. Instead, we report in Table \ref{Number_vertices} the selected values of $n$ and the number $\mathcal{V}_n$ of the vertices of $\widetilde{U}_n$ when $\rho=1$. The numerical results 
show that $n_0=19$ for this case. In our experiment, we begin the computation with an initial integer equal to 12. Note that associated with different parameters $\rho>0$, the polytopes are similar, and thus, the values of $n_0$ for the polytopes are the same. 
With the number  $n_0=19$, we now rewrite the dual problem \eqref{numerical_dual} as 
\begin{equation}\label{numerical_dual_finite}
    \sup\left\{\sum_{j\in\mathbb{N}_{12}}y_j\lambda_j : |\lambda_j|\leq 1,\ j\in\mathbb{N}_{12}, \left|\sum_{j\in\mathbb{N}_{12}}\lambda_ja_{j,k}\right|\leq \rho, k\in\mathbb{N}_{19}   \right\}.
\end{equation}
Problem \eqref{numerical_dual_finite} is a typical linear programming problem. We then solve  \eqref{numerical_dual_finite} by using the function ``linprog'' available in Matlab to obtain its numerical solution $\hat\lambda$ and the supremum $S$ of \eqref{numerical_dual_finite}.

\begin{table}[ht]
\caption{The number of the vertices}
\vspace*{0.5cm}
\centering
\setlength{\tabcolsep}{2mm}
\begin{tabular}{cc|c|c|c|c|c|c|c|cc}  
\hline\hline
&$n$ &14  &16  &18 &19 &20 &21 &25 &30 \\ \hline
&$\mathcal{V}_n$ &$10,256$ &$21,070$ &$44,134$ &$59,930$ &$59,930$ &$59,930$ &$59,930$ &$59,930$\\ 
\hline\hline
\end{tabular}
\label{Number_vertices}
\end{table}

\begin{table}[ht]
\caption{Numerical results for regularized extremal problem with $m=12$.}
\vspace*{0.5cm}
\centering
\setlength{\tabcolsep}{2mm}
\begin{tabular}{lc|c|c|c|c|c|c|c}
\hline\hline
&$\rho$ &$\rho\|\hat{\lambda}\|_{\infty}$  &$\|\mathbf{A}_{*}\hat{\lambda}\|_{\infty}$  &$S$ &$f_r$ &$\mathrm{SL}$  &$\mathrm{ERR}$  &$\|\mathbf{y}-A\hat{\mathbf{x}}\|_2$ \\ \hline
&12.0000 &12.0000 &7.7309 &0.7637 &0.7637 &0  &0.1039 &0.2451\\
\hline
&10.0000 &10.0000 &7.7309 &0.7637 &0.7637 &0 &0.1039 &0.2451\\
\hline
&8.0000 &8.0000 &7.7309 &0.7637 &0.7637 &0 &0.1039 &0.2451 \\
\hline
&7.0000 &7.0000 &7.0000 &0.7180  &0.7180 &1 &0.0336   &0.0525 \\
\hline
&5.0000 &5.0000 &5.0000 &0.5498 &0.5498 &1  &0.0316   &0.0449  \\
\hline
&1.0000 &1.0000 &1.0000 &0.1407 &0.1411 &3   &0.0103   &0.0070  \\
\hline
&0.8000 &0.8000 &0.8000 &0.1145 &0.1147 &4 &0.0076 &0.0044 \\
\hline
&0.5000 &0.5000 &0.5000 &0.0736 &0.0737 &5  &0.0055 &0.0025 \\
\hline
&0.3000 &0.3000 &0.3000 &0.0455 &0.0458 &6  &0.0049 &0.0021 \\
\hline
&0.2500 &0.2500 &0.2500 &0.0384 &0.0387 &7  &0.0051  &0.0020 \\
\hline
&0.2000 &0.2000 &0.2000 &0.0310 &0.0310 &9 &0.0031 &7.2907e-4 \\
\hline
&0.1800 &0.1800 &0.1800 &0.0279 &0.0279 &10 &0.0030 &7.4563e-4 \\
\hline
&0.1000 &0.0660 &0.1000 &0.0155 &0.0155 &12  &0.0032 &5.4190e-4 \\
\hline
&0.0100 &6.6010e-4 &0.0100 &0.0016 &0.0016 &12  &0.0032 &5.4190e-4\\
\hline
&0.0010 &6.6010e-6 &0.0010 &1.5550e-4 &1.5550e-4 &12
&0.0032 &5.4190e-4 \\
\hline\hline
\end{tabular}
\label{Regularization_problem_m12}
\end{table}

\begin{table}[ht]
\caption{Numerical results for regularized extremal problem with $m=200$.}
\vspace*{0.5cm}
\centering
\setlength{\tabcolsep}{2mm}
\begin{tabular}{lc|c|c|c|c|c|c|c}
\hline\hline
&$\rho$ &$\rho\|\hat{\lambda}\|_{\infty}$  &$\|\mathbf{A}_{*}\hat{\lambda}\|_{\infty}$  &$S$ &$f_r$ &$\mathrm{SL}$  &$\mathrm{ERR}$  &$\|\mathbf{y}-A\hat{\mathbf{x}}\|_2$ \\ \hline
&132.0000 &132.0000 &127.0107 &12.8027 &12.8027 &0  &0.1040 &1.0075\\
\hline
&130.0000 &130.0000 &127.0107 &12.8027 &12.8027 &0  &0.1040  &1.0075 \\
\hline
&128.0000 &128.0000 &127.0107 &12.8027 &12.8027 &0  &0.1040  &1.0075 \\
\hline
&127.0000 &127.0000 &127.0000 &12.8026  &12.8026 &1 &0.0915   &0.8778 \\
\hline
&100.0000 &100.0000 &100.0000 &10.7137 &10.7137 &1  &0.0324   &0.1992  \\
\hline
&80.0000 &80.0000 &80.0000 &8.9762 &8.9762 &1   &0.0306   &0.1683  \\
\hline
&50.0000 &50.0000 &50.0000 &6.2516 &6.2516 &2 &0.0212 &0.0979 \\
\hline
&10.0000 &10.0000 &10.0000 &1.4854 &1.4855 &5  &0.0064 &0.0135 \\
\hline
&1.0000 &1.0000 &1.0000 &0.1891 &0.1894 &12  &0.0019  &0.0020 \\
\hline
&0.1000 &0.1000 &0.1000 &0.0413 &0.0453 &36  &0.0077 &0.0019 \\
\hline
&0.0500 &0.0500 &0.0500 &0.0272 &0.0331 &67 &0.0188 &0.0025 \\
\hline
&0.0300 &0.0300 &0.0300 &0.0177 &0.0239 &132  &0.0410 &0.0030 \\
\hline
&0.0100 &0.0050 &0.0100 &0.0060 &0.0060 &200  &0.0453  &0.0031 \\
\hline
&0.0010 &5.0431e-5 &0.0010 &5.9779e-4 &5.9779e-4 &200
&0.0453 &0.0031 \\
\hline
&0.0001 &5.0431e-7 &0.0001 &5.9779e-5 &5.9779e-5 &200
&0.0453  &0.0031\\
\hline\hline
\end{tabular}
\label{Regularization_problem_m200}
\end{table}

With a dual solution $\hat\lambda$ at hand, we then solve the original 
problem \eqref{numerical} according to Propositions \ref{norming_functional_predual_regularization_p=1} and \ref{solution_specific_example2}. To this end, we compute 
$\rho\|\hat\lambda\|_{\infty}$ and $\|A_*\hat\lambda\|_{\infty}$ with $A_*$ being defined by \eqref{A_*}, and 
compare their values. 
If $\rho\|\hat\lambda\|_{\infty} > \|A_*\hat\lambda\|_{\infty}$, 
we take  $\hat{\mathbf{x}}=0$ as the solution of \eqref{numerical} 
by (1) of Proposition \ref{norming_functional_predual_regularization_p=1}. 
If $\rho\|\hat\lambda\|_{\infty} \leq \|A_*\hat\lambda\|_{\infty}$, we identify the index set $\mathbb{N}(\hat\mu)$ of $\hat\mu:=A_*\hat\lambda$, on which the sequence $\hat\mu$ achieves its supremum norm, the cardinality $n_{\hat\mu}$ of  $\mathbb{N}(\hat\mu)$, and generate the matrix $A_{\hat\mu}$ defined by \eqref{trancate_matrix} with $\mathbf{c}$ being replaced by $\hat\mu$. According to Proposition \ref{solution_specific_example2}, we solve the finite dimensional optimization problem \begin{equation}\label{numerical_finite}
     \inf \big\{\|\mathbf{y}_0-A_{\hat\mu}\mathbf{z}\|_{1}+\rho \|\mathbf{z}\|_{1}:\ \mathbf{z} \in \mathbb{R}^{n_{\hat\mu}} \big\}, 
     \end{equation}
by employing the FPPA originally developed in \cite{LSXZ, MSX}. We describe the FPPA as follows: Let $f:\mathbb{R}^d\to  \mathbb{R}\cup\{+\infty\}$ be a convex function such that $\mathrm{dom}(f):=\{\mathbf{w}\in\mathbb{R}^d:f(\mathbf{w})<+\infty\}\neq{\emptyset}.$ The proximity operator $\prox_{f}:\mathbb{R}^d\to\mathbb{R}^d$ of a convex function $f$ is defined for $\mathbf{w}\in\mathbb{R}^d$ by
\begin{equation*}
\prox_{f}(\mathbf{w}):=\argmin\left\{\frac{1}{2}\|\mathbf{u}-\mathbf{w}\|_2^2+f(\mathbf{u}):\mathbf{u}\in\mathbb{R}^d\right\}.
\end{equation*}
Set $\varphi:=\rho\|\cdot\|_1$ and $\psi:=\|\mathbf{y}_0-\cdot\|_1$. 
By choosing positive constants $\beta$, $\gamma$ and initial points $\mathbf{z}^0\in\mathbb{R}^{n_{\mu}}$, $\mathbf{v}^0\in\mathbb{R}^m$, we solve \eqref{numerical_finite} by the FPPA
\begin{equation}\label{FPPA}
\left\{\begin{array}{l}
\mathbf{z}^{k+1}=\operatorname{prox}_{\beta\varphi}\left(\mathbf{z}^{k}-\beta A_{\hat\mu}^{\top} \mathbf{v}^{k}\right), \\
\mathbf{v}^{k+1}=\gamma\left(\mathcal{I}-\operatorname{prox}_{\frac{1}{\gamma}\psi}\right)\left(\frac{1}{\gamma} \mathbf{v}^{k}+A_{\hat\mu}\left(2 \mathbf{z}^{k+1}-\mathbf{z}^{k}\right)\right),
\end{array}\right.
\end{equation}
to obtain a numerical solution $\hat{\mathbf{z}}$. In Algorithm \eqref{FPPA}, 
parameters $\beta$ and $\gamma$ are chosen so that the algorithm converges, 
and $\operatorname{prox}_{\beta\varphi}$ and 
$\operatorname{prox}_{\frac{1}{\gamma}\psi}$ have closed-forms that we 
present below.
The proximity operator $\prox_{\beta\varphi}$ at $\mathbf{w}:=(w_j:j\in\mathbb{N}_{n_{\hat\mu}})\in\mathbb{R}^{n_{\hat\mu}}$ has the form $
\prox_{\beta\varphi}(\mathbf{w}):=(u_j:j\in\mathbb{N}_{n_{\hat\mu}}), 
$
where for all $j\in\mathbb{N}_{n_{\hat\mu}}$
\begin{equation*}
u_j:=
\left\{\begin{array}{ll}
w_j-\beta\rho, \ & \mbox{if}\ w_j>\beta\rho, \\
w_j+\beta\rho, \  & \mbox{if}\ w_j<-\beta\rho ,\\
0, \ & \mbox{if}\ w_j\in[-\beta\rho,\beta\rho].
\end{array}\right.
\end{equation*}
Likewise, the proximity operator $\prox_{\frac{1}{\gamma}\psi}$ at $\mathbf{w}:=(w_j:j\in\mathbb{N}_{m})\in\mathbb{R}^{m}$ has the form
$\prox_{\frac{1}{\gamma}\psi}(\mathbf{w}):=(u_j:j\in\mathbb{N}_{m})$, where for all $j\in\mathbb{N}_{m}$
\begin{equation*}
u_j:=
\left\{\begin{array}{ll}
w_j-{1/\gamma}, \ & \mbox{if}\ w_j>y_j+{1/\gamma}, \\
w_j+{1/\gamma}, \  & \mbox{if}\ w_j<y_j-{1/\gamma},\\
y_j, \ & \mbox{if}\ w_j\in[y_j-{1/\gamma},y_j+{1/\gamma}].
\end{array}\right.
\end{equation*}
Algorithm FPPA generates a numerical solution $\hat{\mathbf{z}}$ of \eqref{numerical_finite}, with which a numerical solution $\hat{\mathbf{x}}$ of the original problem \eqref{numerical} is obtained by augmenting $\hat{\mathbf{z}}$ as described in Proposition \ref{solution_specific_example2}. 

\begin{table}[ht]
\caption{Numerical results for regularized extremal problem with $m=600$.}
\vspace*{0.5cm}
\centering
\setlength{\tabcolsep}{2mm}
\begin{tabular}{lc|c|c|c|c|c|c|c}
\hline\hline
&$\rho$ &$\rho\|\hat{\lambda}\|_{\infty}$  &$\|\mathbf{A}_{*}\hat{\lambda}\|_{\infty}$  &$S$ &$f_r$ &$\mathrm{SL}$  &$\mathrm{ERR}$  &$\|\mathbf{y}_0-A\hat{\mathbf{x}}\|_2$ \\ \hline
&385.0000 &385.0000 &380.3394 &38.3914 &38.3914 &0  &0.1040 &1.7456\\
\hline
&382.0000 &382.0000 &380.3394 &38.3914 &38.3914 &0  &0.1040 &1.7456 \\
\hline
&381.0000 &381.0000 &380.3394 &38.3914 &38.3914 &0  &0.1040 &1.7456 \\
\hline
&380.0000 &380.0000 &380.0000 &38.3864 &38.3864 &1  &0.0841  &1.3868 \\
\hline
&300.0000 &300.0000 &300.0000 &32.1429 &32.1429 &1  &0.0323  &0.3412  \\
\hline
&250.0000 &250.0000 &250.0000 &27.8101 &27.8101 &1   &0.0311 &0.3070  \\
\hline
&150.0000 &150.0000 &150.0000 &18.7604 &18.7604 &2 &0.0216 &0.1761 \\
\hline
&50.0000 &50.0000 &50.0000 &7.1210 &7.1210 &3  &0.0105 &0.0509 \\
\hline
&10.0000 &10.0000 &10.0000 &1.6264 &1.6265 &7  &0.0033 &0.0075 \\
\hline
&1.0000 &1.0000 &1.0000 &0.2608 &0.2616 &17  &0.0012  &0.0014 \\
\hline
&0.1000 &0.1000 &0.1000 &0.1071 &0.1151 &69 &0.0121 &0.0025 \\
\hline
&0.0500 &0.0500 &0.0500 &0.0837 &0.0988 &150  &0.0369 &0.0036\\
\hline
&0.0200 &0.0200 &0.0200 &0.0413 &0.0487 &515  &0.1132 &0.0053 \\
\hline
&0.0100 &0.0073 &0.0100 &0.0208 &0.0208 &600  &0.1146  &0.0054 \\
\hline
&0.0010 &7.3254e-5 &0.0010 &0.0021 &0.0021 &600
&0.1146 &0.0054 \\
\hline
&0.0001 &7.3254e-7 &0.0001 &2.0760e-4 &2.0801e-4 &600
&0.1146  &0.0054\\
\hline\hline
\end{tabular}
\label{Regularization_problem_m600}
\end{table}

For convenience, we use $f_{r}$ to denote the value of the objective function $f:=\|\mathbf{y}_0-A(\cdot)\|_{1}+\rho \|\cdot\|_{1}$ at the numerical solution $\hat{\mathbf{x}}$ of \eqref{numerical}, which is a computed infimum of \eqref{numerical}. We define  $\mathrm{ERR}:=\|\hat{\mathbf{x}}-\mathbf{x}^\dag\|_2$, where $\mathbf{x}^\dag=\argmin\{\|\mathbf{x}\|_1:A\mathbf{x}=\mathbf{y}, \mathbf{x}\in\ell_1(\mathbb{N})\}$. Here, we solve the minimum norm interpolation problem by the duality approach developed in \cite{CX1}. In Table \ref{Regularization_problem_m12} we report the selected values of $\rho$, the corresponding values of $\rho\|\hat\lambda\|_{\infty}$, $\|A_{*}\hat\lambda\|_{\infty}$, $S$, $f_{r}$, ERR, $\|\mathbf{y}-A\hat{\mathbf{x}}\|_2$ and the sparsity levels SL of $\hat{\mathbf{x}}$. From the numerical results, we observe that the value of $f_{r}$ approximates the supremum $S$ (which is equal to the infimum of \eqref{numerical}) very well. This shows that the proposed dual approach provides an effective numerical method for solving the regularized extremal problem \eqref{numerical}.

In the second and the third experiments, we solve the regularized extremal problem \eqref{numerical} with $m=200$ and $m=600$, respectively, by the duality approach. This time we increase the size of the problem from $m=12$ to $m=200$ and $m=600$. We choose the number $n_0$ of the constraints in the dual problem \eqref{numerical_dual} as $n_0=333$ and $n_0=710$ by using the second empirical method described earlier. The selected values of $\rho$, the values of $\rho\|\hat\lambda\|_{\infty}$, $\|A_{*}\hat\lambda\|_{\infty}$, $S$, $f_{r}$, ERR, $\|\mathbf{y}-A\hat{\mathbf{x}}\|_2$ and the sparsity levels SL of $\hat{\mathbf{x}}$ are reported in Table \ref{Regularization_problem_m200} and Table \ref{Regularization_problem_m600}. The numerical results indicate that for problems of these sizes, the proposed duality approach works well.

The three numerical experiments presented in this section demonstrate that the 
proposed duality approach is feasible for solving the regularization problem 
in infinite dimensional Banach spaces. 

To close this section, we remark that the main purpose of this paper is 
to lay out the mathematical foundation of the duality approach. Issues 
related to practical implementation of this approach remain to be addressed.
Addressing the issues will be our future research project. 

\vspace{-3mm}
\section*{Appendix: Proofs of Auxiliary Results}
\setcounter{equation}{0}
\renewcommand\theequation{A.\arabic{equation}}

Throughout the paper, numerous standard or straightforward results from functional analysis are used.  Their proofs are collected here for reference.
\vspace{5mm}

\noindent {\bf Proof of Proposition \ref{extvalatt}}\ \,
  There exists a sequence $\mathbf{a}'_n \in \mathscr{N}^{\perp}$, $n\in\mathbb{N}$, such that
\begin{equation}\label{infimum}
     \lim_{n\rightarrow\infty} \|\mathbf{a} - \mathbf{a}'_n\|_{\mathscr{A}} = \dist(\mathbf{a},\mathscr{N}^{\perp}).
\end{equation}
The sequence  $\{\mathbf{a}'_n\}_{n=1}^{\infty}$ is bounded in norm, since there holds
$\|\mathbf{a}'_n\|_{\mathscr{A}} \leq \|\mathbf{a}\|_{\mathscr{A}} +\|\mathbf{a} - \mathbf{a}'_n\|_{\mathscr{A}}.$
Therefore the Banach-Alaoglu Theorem supplies a subsequence $\{\mathbf{a}'_{n_k}:k\in\mathbb{N}\}$ that converges in the weak${}^*$ sense to some $\mathbf{a}' \in \mathscr{A}$.  That is, $  \lim_{k\rightarrow\infty}\mathbf{a}_{n_k}(\ell)  = \mathbf{a}'(\ell)
$, for all $\ell\in\mathscr{A}_*$. In particular, if $\ell\in \mathscr{N}$, then $\mathbf{a}'(\ell) = \lim_{k\rightarrow\infty}\mathbf{a}'_{n_k}(\ell) = 0$. Hence, $\mathbf{a}'\in \mathscr{N}^{\perp}$. It suffices to verify that $\|\mathbf{a} - \mathbf{a}' \|_{\mathscr{A}} = \dist(\mathbf{a},\mathscr{N}^{\perp})$. By the definition of the norm of $\mathscr{A}$, we have that for any $\epsilon>0$ there exists a unit vector $\ell\in\mathscr{A}_*$ such that \begin{equation}\label{equality}
    |\mathbf{a}(\ell) - \mathbf{a}'(\ell)| \geq \|\mathbf{a} - \mathbf{a}'\|_{\mathscr{A}} - \epsilon.
\end{equation}
It follows from equation \eqref{infimum} that
$\dist(\mathbf{a},\mathscr{N}^{\perp}) 
=\lim_{k\rightarrow\infty}\|\mathbf{a} - \mathbf{a}'_{n_k}\|_{\mathscr{A}}\geq \lim_{k\rightarrow\infty} |\mathbf{a}(\ell) - \mathbf{a}'_{n_k}(\ell)|, $
which further leads to 
$\dist(\mathbf{a},\mathscr{N}^{\perp})=|\mathbf{a}(\ell) - \mathbf{a}'(\ell)|.$
Substituting inequality \eqref{equality} into the above equation, we obtain that $
\dist(\mathbf{a},\mathscr{N}^{\perp}) 
\geq \|\mathbf{a} - \mathbf{a}'\|_{\mathscr{A}}-\epsilon.$
Since $\epsilon$ was arbitrary, it follows that $\dist(\mathbf{a},\mathscr{N}^{\perp}) \geq\|\mathbf{a} - \mathbf{a}'\|_{\mathscr{A}}$.  The reverse inequality holds because $\mathbf{a}' \in \mathscr{N}^{\perp}$, and thus the claim is proved.

\vspace{5mm}

\vspace{5mm}
\noindent{\bf Proof of Proposition \ref{normfcnldirsum}}\ \,
Since $(\mathbf{a},\mathbf{b})\in{\mathscr{A} \oplus_p \mathscr{B}}$ is normed by $(\lambda,\mu) \in \mathscr{A}^*\oplus_{p'} \mathscr{B}^*$, we have that
   \begin{equation}\label{lambdamu_norm_ab}
       \big(\|\lambda\|_{\mathscr{A}^*}^{p'}+\|\mu\|^{p'}_{\mathscr{B}^*}\big)^{1/{p'}} = 1,\  \ \langle \mathbf{a}, \lambda \rangle_{\mathscr{A}} + \langle \mathbf{b},  \mu \rangle_{\mathscr{B}}
       = \big(\|\mathbf{a}\|^p_{\mathscr{A}}+ \|\mathbf{b}\|^p_{\mathscr{B}} \big)^{1/p}.
  \end{equation}
   At the same time, it must be that 
   \begin{align}
     {\langle \mathbf{a},\lambda\rangle_{\mathscr{A}} + \langle \mathbf{b},\mu\rangle_{\mathscr{B}}}
     &\leq \|\mathbf{a}\|_{\mathscr{A}}\|\lambda\|_{\mathscr{A}^*}+\|\mathbf{b}\|_{\mathscr{B}}\|\mu\|_{\mathscr{B}^*}\nonumber\\
     &\leq \big(\|\mathbf{a}\|_{\mathscr{A}}^p+\|\mathbf{b}\|_{\mathscr{B}}^p\big)^{1/p}\big(\|\lambda\|_{\mathscr{A}^*}^{p'}+\|\mu\|^{p'}_{\mathscr{B}^*}\big)^{1/{p'}}\nonumber\\
     &= \big(\|\mathbf{a}\|_{\mathscr{A}}^p+\|\mathbf{b}\|_{\mathscr{B}}^p\big)^{1/p}.\label{normest3}
  \end{align} 
  Equality is forced throughout \eqref{normest3}.  In particular, the condition for equality must hold in H\"{o}lder's inequality, as employed in the second step of \eqref{normest3}. The following identifications result 
  \begin{equation}\label{normest4}
  \langle \mathbf{a},\lambda\rangle_{\mathscr{A}}=\|\mathbf{a}\|_{\mathscr{A}}\|\lambda\|_{\mathscr{A}^*},\ \langle \mathbf{b},\mu\rangle_{\mathscr{B}}= \|\mathbf{b}\|_{\mathscr{B}}\|\mu\|_{\mathscr{B}^*}.
  \end{equation}
  and 
  \begin{equation}\label{normest5}
  \|\mathbf{a}\|^p_{\mathscr{A}} = C\|\lambda\|^{p'}_{\mathscr{A}^*},\ 
     \|\mathbf{b}\|^p_{\mathscr{B}} = C\|\mu\|^{p'}_{\mathscr{B}^*},
  \end{equation}
 for some positive constant $C$.  In fact, the value of $C$ is determined from 
$\|\mathbf{a}\|^p_{\mathscr{A}}+\|\mathbf{b}\|^p_{\mathscr{B}} = C\big(\|\lambda\|^{p'}_{\mathscr{A}^*}+\|\mu\|^{p'}_{\mathscr{B}^*}\big)= C.$
 
If $\mathbf{b}=0$, then the second equation in \eqref{normest5} yields that $\mu=0$, which together with the first equation in \eqref{lambdamu_norm_ab} leads to $\|\lambda\|_{\mathscr{A}^*}=1$. Hence, we conclude from the first equation in \eqref{normest4} that $\lambda$ is norming for $\mathbf{a}$. This proves statement 1. We can prove statement 2 by similar arguments. Finally, we verify statement 3. If $\mathbf{a}\neq 0$, $\mathbf{b}\neq 0$, equations in \eqref{normest5} show that $\lambda$ and $\mu$ are both nonzero. By normalizing the functionals $\lambda$ and $\mu$, we obtain norming functionals for $\mathbf{a}$ and $\mathbf{b}$, respectively:
 \begin{align*}
      \|\lambda\|_{\mathscr{A}^*}  
      &=  C^{-1/{p'}} \|\mathbf{a}\|_{\mathscr{A}}^{p/{p'}} =   \big(1+ \|\mathbf{b}\|^p_{\mathscr{B}}/\|\mathbf{a}\|^p_{\mathscr{A}} \big)^{-1/{p'}}, \\
     \|\mu\|_{\mathscr{B}^*}  
     &=  C^{-1/{p'}} \|\mathbf{b}\|_{\mathscr{B}}^{p/{p'}} =  \big(1+ \|\mathbf{a}\|^p_{\mathscr{A}}/\|\mathbf{b}\|^p_{\mathscr{B}} \big)^{-1/{p'}}.
 \end{align*}
  This proves the claims.
 
\vspace{5mm}
\noindent {\bf Proof of Proposition \ref{normfcnldirsum_p=infty}}\ \,
     By the hypothesis that $(\lambda,\mu)\in \mathscr{A}^* \oplus_{1} \mathscr{B}^*$ is norming for $(\mathbf{a},\mathbf{b})\in \mathscr{A} \oplus_{\infty} \mathscr{B}$, we have that
     \begin{equation}\label{lambdamu_norm_ab_infty}
         \|\lambda\|_{\mathscr{A}^*} + \|\mu\|_{\mathscr{B}^*}  = 1,
           \ \ \ 
           \langle\mathbf{a},\lambda\rangle_{\mathscr{A}}+ \langle\mathbf{b},\mu\rangle_{\mathscr{B}} =\max\{\|\mathbf{a}\|_{\mathscr{A}},\|\mathbf{b}\|_{\mathscr{B}}\}.
     \end{equation}
     According to the first equation in \eqref{lambdamu_norm_ab_infty}, we get that 
     \begin{align}
     {\langle \mathbf{a},\lambda\rangle_{\mathscr{A}} + \langle \mathbf{b},\mu\rangle_{\mathscr{B}}}
     &\leq \|\mathbf{a}\|_{\mathscr{A}}\|\lambda\|_{\mathscr{A}^*}+\|\mathbf{b}\|_{\mathscr{B}}\|\mu\|_{\mathscr{B}^*}\nonumber\\
     &\leq \big(\|\lambda\|_{\mathscr{A}^*} + \|\mu\|_{\mathscr{B}^*}\big)\max\{\|\mathbf{a}\|_{\mathscr{A}},\|\mathbf{b}\|_{\mathscr{B}}\}\nonumber\\
     &= \max\{\|\mathbf{a}\|_{\mathscr{A}},\|\mathbf{b}\|_{\mathscr{B}}\},\label{normest3_infty}
     \end{align} 
     which together with the second equation in \eqref{lambdamu_norm_ab_infty} shows that equality is forced throughout \eqref{normest3_infty}. Hence, we obtain equations in \eqref{normest4} and \begin{equation}\label{normest4_infty}
     \|\mathbf{a}\|_{\mathscr{A}}\|\lambda\|_{\mathscr{A}^*}=\|\lambda\|_{\mathscr{A}^*}\max\{\|\mathbf{a}\|_{\mathscr{A}},\|\mathbf{b}\|_{\mathscr{B}}\}, \ \  
     \|\mathbf{b}\|_{\mathscr{B}}\|\mu\|_{\mathscr{B}^*}=\|\mu\|_{\mathscr{B}^*}\max\{\|\mathbf{a}\|_{\mathscr{A}},\|\mathbf{b}\|_{\mathscr{B}}\}.
     \end{equation}
     
    To prove statement 1, we suppose that $\|\mathbf{a}\|_{\mathscr{A}} > \|\mathbf{b}\|_{\mathscr{B}}$. 
    Then the second equation in \eqref{normest4_infty} reduces to $\|\mathbf{b}\|_{\mathscr{B}}\|\mu\|_{\mathscr{B}^*}=\|\mathbf{a}\|_{\mathscr{A}}\|\mu\|_{\mathscr{B}^*}$,  which further yields that $\mu=0$. We then conclude by the equations in \eqref{lambdamu_norm_ab_infty} that $ \|\lambda\|_{\mathscr{A}^*}=1$ and $\langle\mathbf{a},\lambda\rangle_{\mathscr{A}} =\|\mathbf{a}\|_{\mathscr{A}}$, that is, $\lambda$ is norming for $\mathbf{a}$. This completes the proof of statement 1 and statement 2 may be proved similarly. It remains to show statement 3. We assume that $\|\mathbf{a}\|_{\mathscr{A}} = \|\mathbf{b}\|_{\mathscr{B}}$. If  $\mu=0$, then equations in \eqref{lambdamu_norm_ab_infty} lead directly to $ \|\lambda\|_{\mathscr{A}^*}=1$ and $\langle\mathbf{a},\lambda\rangle_{\mathscr{A}} =\|\mathbf{a}\|_{\mathscr{A}}$. That is to say, $\lambda$ is norming for $\mathbf{a}$. likewise, if $\lambda=0$, we may show that $\mu$ is norming for $\mathbf{b}$. Finally, if  $\lambda$ and $\mu$ are both nonzero vectors, then we conclude by  equations in \eqref{normest4} that $\lambda/\|\lambda\|_{\mathscr{A}^*}$ is norming for $\mathbf{a}$, and $\mu/\|\mu\|_{\mathscr{B}^*}$ is norming for $\mathbf{b}$, proving the desired results.

\section*{Acknowledgement}

The authors are grateful to Dr. Qianru Liu for assistance in obtaining the numerical results presented in section 5.
R. Wang is supported in part by the National Key Research and Development Program of China (grants no. 2020YFA0714100 and 2020YFA0713600) and by the Natural Science Foundation of China under grant 12171202; Y. Xu is supported in part by the US National Science Foundation under grants DMS-1912958 and DMS-2208386, and by the US National Institutes of Health under grant R21CA263876. All correspondence should be sent to Y. Xu.


\begin{thebibliography}{9999}
%
%
%
%
%


%
\bibitem{Ar}
{\sc A.\ Argyriou, C.\ A.\ Micchelli and M.\ Pontil}, When is there a representer theorem? Vector
versus matrix regularizers, {\em Journal of Machine Learning Research} {\bf 10} (2009), 2507--2529.
%
\bibitem{Az} {\sc S.\ Aziznejad  and M.\ Unser}, Multikernel regression with sparsity constraint, {\em
SIAM Journal on Mathematics of Data Science} {\bf 3} (2021), 201--224. 
%
\bibitem{Barbu-Precupanu}  {\sc V.\ Barbu and T.\ Precupanu}, {\em Convexity and Optimization in Banach Spaces}. Fourth edition.  Springer, Dordrecht, 2012. 
%
\bibitem{bi2003dimensionality}
 {\sc J.\ Bi, K.\ P.\ Bennett, M.\  Embrechts, C.\ M.\ Breneman and M. Song },  Dimensionality reduction via sparse support vector machines, {\em Journal of Machine Learning Research} {\bf 3} (2003), 1229--1243.
%
\bibitem{BP} {\sc H. Boche and V. Pohl}, The stability and continuity behavior of the spectral factorization in the Wiener algebra with applications in Wiener filtering, {\em IEEE Transactions on Circuits and Systems. I. Regular Papers} {\bf 55} (2008), 3063--3076.
%
\bibitem{Candes} {\sc E.\ J.\ Cand\'es, J.\ Romberg and T.\ Tao},
Robust uncertainty principles: Exact signal reconstruction from highly incomplete frequency information, 
{\em IEEE Transactions on Information Theory} {\bf 52} (2006), 489--509.


%
%
\bibitem{chen2001atomic}
{\sc S.\ S.\ Chen,  D.\ L.\ Donoho and  M.\ A.\ Saunders},  Atomic decomposition by basis pursuit, {\em SIAM Review} {\bf 43} (2001), 129--159.
%
\bibitem{Ch} {\sc R.\ Cheng},  On the prediction of $p$-stationary processes, {\em Periodica Mathematica Hungarica} {\bf 85} (2022), 481--505. 
%
\bibitem{CR} {\sc R.\ Cheng and W.\ T.\ Ross},  Weak parallelogram laws on Banach spaces and applications to prediction, {\em Periodica Mathematica Hungarica} {\bf 71} (2015), 45--58.
%
\bibitem{CR2}  {\sc R.\ Cheng and W.\ T.\ Ross},  An inner-outer factorization in $\ell^p$ with applications to ARMA processes, {\em Journal of Mathematical Analysis and Applications
} {\bf 437} (2016), 396--418.
%
\bibitem{CX1} {\sc R.\ Cheng and Y.\ Xu}, Minimum norm interpolation in the $\ell_1(\mathbb{N})$ space, {\em Analysis and Applications} {\bf 19} (2021), 21--42.
%
%
\bibitem{Con} {\sc J.\ Conway},  {\em A Course in Functional Analysis}. Springer-Verlag, New York, 1985.
%
%
%
\bibitem{CS} {\sc
F.\ Cucker and S.\ Smale}, On the mathematical foundations of learning, {\em Bulletin of the AMS} {\bf 39} (2002), 1--49.
%
\bibitem{Donoho} {\sc D.\,L.\ Donoho}, Compressed Sensing, {\em IEEE Transactions on Information Theory} {\bf 52} (2006), 1289--1306.
%
\bibitem{EPP} 
{\sc T.\ Evgeniou, M.\ Pontil and T.\ Poggio}, Regularization networks and support vector machines,
{\em Advances in Computational Mathematics} {\bf 13} (2000), 1--50.
%
\bibitem{HLTY} 
{\sc  L. Huang, C. Liu, L. Tan and Q. Ye}, Generalized representer theorems in Banach spaces,
{\em Analysis and Applications} {\bf 19} (2021), no. 1, 125--146.
%
%
\bibitem{Kur} {\sc E. E. Kuruo\H{g}lu},  Nonlinear least $\ell^p$ norm filters for nonlinear autoregressive $\alpha$-stable processes, {\em Digital Signal Processing} {\bf 12} (2002), 119--142.
%
\bibitem{LSXZ} 
 {\sc Q.\ Li, L.\ Shen, Y.\ Xu and N.\ Zhang}, Multi-step fixed-point proximity algorithms for solving a class of optimization problems arising from image processing,
{\em Advances in Computational Mathematics} {\bf 41} (2015), 387--422.
%
\bibitem{LSX} {\sc Z.\ Li, G.\ Song and Y.\ Xu},
A two-step fixed-point proximity algorithm for a class of non-differentiable optimization models in machine learning,
{\em Journal of Scientific Computing} {\bf 81} (2019), 923--940.
%
\bibitem{lin2021multi}
{\sc R.\ Lin, G. Song and H. Zhang}, Multi-task Learning in vector-valued reproducing kernel Banach spaces with the $\ell_1$ norm, {\em Journal of Complexity} {\bf 63} (2021), 101514.

%
\bibitem{LWXY} {\sc Q.\ Liu, R.\ Wang, Y.\ Xu and M. Yan}, Parameter choices for sparse regularization with the $\ell_1$ Norm, {\em Inverse Problem} {\bf 39} (2023), 025004. 
%
\bibitem{LP13} {\sc S.\ Lu and S.\  V.\ Pereverzev}, {\em Regularization Theory for Ill-posed Problems}. Selected Topics. De Gruyter, Berlin, Boston, 2013.
%
\bibitem{Meg} {\sc R.\ E.\ Megginson}, {\em An Introduction to Banach Space Theory}. Graduate Texts in Mathematics, vol. 183. Springer-Verlag, New York, 1998.
%
\bibitem{MSX} {\sc C.\ A.\ Micchelli, L.\ Shen and Y.\ Xu}, Proximity algorithms for image models: denoising, 
{\em Inverse Problems} {\bf 27} (2011), 045009.
%
\bibitem{MXY} {\sc C.\ A.\ Micchelli, Y.\ Xu and P.\ Ye},
Cucker Smale learning theory in Besov spaces. Advances in
Learning Theory: Methods, Models and Applications. J. Suykens, G. Horvath, S. Basu, C. A.
Micchelli and J. Vandewalle, editors. IOS Press, Amsterdam, The Netherlands, 2003, 47--68.
%
%
%
\bibitem{PN} 
{\sc R.\ Parhi and R.\ D.\ Nowak}, Banach space representer theorems for neural networks and ridge splines, {\em Journal of Machine Learning Research} {\bf 22} (2021), 43.
%
%
%

\bibitem{rudin1992nonlinear} {\sc L.\ I.\ Rudin,  S.\ Osher and E.\ Fatemi}, 
Nonlinear total variation based noise removal algorithms,
{\em Physica D. Nonlinear Phenomena} {\bf 60} (1992), 259--268.
%

\bibitem{scholkopf2002learning} {\sc B.\ Sch\"olkopf  and A.\ J.\ Smola}, {\em Learning with Kernels: Support Vector Machines, Regularization, Optimization, and Beyond}. MIT Press, Cambridge, 2002.
%

\bibitem{scholkopf2001} {\sc B.\ Sch\"olkopf, R. Herbrich  and A.\ J.\ Smola}, {\em
A generalized representer theorem}, in ``Computational Learning Theory'', D. Helmbold, and B. Williamson (eds.), Lecture Notes in Computer Science Vol. 2111,  Springer, Berlin,  416--426. 


\bibitem{SKHK} 
{\sc T. Schuster, B. Kaltenbacher, B. Hofmann and K. S. Kazimierski}, {\em Regularization Methods in Banach Spaces}. Springer, Berlin, 2012.
%
%
\bibitem{Song-Zhang} 
{\sc G.\ Song and H.\ Zhang}, {
Reproducing kernel Banach spaces with the $\ell^1$ norm II: Error analysis for regularized least square regression},
{\em Neural Computation} {\bf 23} (2011), 2713--2729.
%
\bibitem{Song-Zhang-Hickernell} {\sc G.\ Song, H.\ Zhang and F.\,J.\ Hickernell},
Reproducing kernel Banach spaces with the $\ell^1$ norm,
{\em Applied and Computational Harmonic Analysis} {\bf 34} (2013), 96--116.
%
\bibitem{Sri} {\sc B.\ K.\ Sriperumbudur, K.\ Fukumizu and G.\ R.\ G.\ Lanckriet},
Learning in Hilbert vs. Banach spaces: A measure embedding viewpoint, {\em NIPS}, 2011.
%

\bibitem{tibshirani1996regression}
{\sc R.\ Tibshirani}, Regression shrinkage and selection via the lasso, {\em Journal of the Royal Statistical Society. Series B. Methodological} {\bf 58} (1996), 267--288.
%

\bibitem{tibshirani2011solution}
{\sc R.\ J.\ Tibshirani and J.\ Taylor}, 
The solution path of the generalized lasso, {\em  The Annals of Statistics} {\bf 39} (2011), 1335--1371.
%

\bibitem{Unser2} {\sc M.\ Unser}, A unifying representer theorem for inverse problems and machine learning, {\em Foundations
of Computational Mathematics} {\bf 21} (2021), 941--960.
%
\bibitem{Unser1} {\sc M.\ Unser, J.\ Fageot and H.\ Gupta}, Representer theorems for sparsity-promoting $\ell_1$ regularization, {\em IEEE Transactions on Information Theory} {\bf 62} (2016), 5167--5180.
%
%
\bibitem{Ubhaya1990}  
{\sc V.\ A.\ Ubhaya, S.\ E.\ Weinstein and Y.\ Xu}, 
Best piecewise monotone uniform approximation, {\em Journal of Approximation Theory} {\bf 63} (1990) 375--383.
%
\bibitem{VZ} {\sc C. Viola and S. \u{Z}ivn\'{y}},  The combined basic $L^p$ and affine $\ell^p$ relaxation for promise VCSPs on infinite domains, {\em ACM Transactions on Algorithms} {\bf 17} (2021), Art. 21, 23 pp.
%
\bibitem{WX2019} 
{\sc R.\ Wang and Y.\ Xu},  Functional reproducing kernel Hilbert spaces for non-point-evaluation functional data,
{\em Applied and Computational Harmonic Analysis} {\bf 46} (2019), 569--623.
%
\bibitem{WX2020} 
{\sc R.\ Wang and Y.\ Xu},  Representer theorems in Banach spaces: Minimum norm interpolation, regularized
learning and semi-discrete inverse problems, {\em
Journal of Machine Learning Research} {\bf 22} (2021), 225.
%
\bibitem{WX2021} 
{\sc R.\ Wang and Y.\ Xu}, Regularization in a functional reproducing kernel Hilbert space, {\em Journal of Complexity} {\bf 66} (2021), 101567.
%
\bibitem{Weinstein1990} 
{\sc S.\ E.\ Weinstein and Y.\ Xu}, Best quasi-convex uniform approximation, {\em Journal of Mathematical Analysis and Applications} {\bf 152} (1990), 240--251.
%
\bibitem{Weinstein1991}  
{\sc S.\ E.\ Weinstein and Y.\ Xu}, A duality approach to best uniform convex approximation, {\em Journal of Mathematical Analysis and Applications} {\bf 160} (1991), 314--322. 
%
\bibitem{Xu2023} {\sc  Y. \ Xu}, Sparse machine learning in Banach spaces, {\em Applied Numerical Mathematics}
{\bf 187} (2023), 138-157.

\bibitem{XuYe} 
{\sc Y.\ Xu and Q.\ Ye}, Generalized Mercer kernels and reproducing kernel Banach spaces, {\em Memoirs of the American Mathematical Society} {\bf 258} (2019), 1243.
%
%
\bibitem{ZSZ} {\sc W. J. Zeng, H. C. So and A. M. Zoubir},  An $\ell_p$-norm minimization approach to time delay estimation in impulsive noise, {\it Digital Signal Processing} {\bf 23} (2013), 1247--1254.
%
\bibitem{Zhang-Xu-Zhang} 
{\sc H.\ Zhang, Y.\ Xu and J.\ Zhang}, Reproducing kernel Banach spaces for machine learning, {\em
Journal of Machine Learning Research} {\bf 10} (2009), 2741--2775.
%








\end{thebibliography}
\end{document}